\documentclass[11pt,reqno]{amsart}
\usepackage[margin=3cm]{geometry}
\usepackage{color}               
\usepackage{faktor}
\usepackage{graphicx}
\usepackage{caption}
\usepackage{subcaption}
\usepackage{tikz}
\usepackage{tikz-cd}
\usepackage{spverbatim}
\usepackage{float}
\usepackage[pdftex,hyperfootnotes=false,colorlinks=false,hypertexnames=false]{hyperref}

\usepackage{amsmath,amssymb,anyfontsize,enumerate,stmaryrd,mathtools, thmtools,tikz,txfonts,nccmath}
\usepackage[oldstyle]{libertine}%
\usepackage[T1]{fontenc}
\usepackage[final]{microtype}
\usepackage[utf8]{inputenc}
\usepackage{csquotes}
\makeatletter
\renewcommand*\libertine@figurestyle{LF}
\makeatother
\usepackage[libertine,libaltvw,liby]{newtxmath}
\makeatletter
\renewcommand*\libertine@figurestyle{OsF}
\makeatother
\usepackage[noerroretextools,backend=biber,style=alphabetic,maxalphanames=5,giveninits,sorting=nyt,%
maxbibnames=99,isbn=false]{biblatex}
\usepackage{tikz-cd}
\usepackage[noabbrev,capitalise]{cleveref}
\usepackage{autonum}
\graphicspath{{Figures/}}

\newcommand{\tPC}{\mathbb{T}\mathbb{P}^{1}(\mathbb{C})}

\newcommand{\tr}{^{\text{trop}}}

\DeclareMathOperator{\p}{_{\text{P}}}
\DeclareMathOperator{\val}{val}

\theoremstyle{plain}
    \newtheorem{theorem}{Theorem}[section]
    \newtheorem{construction/theorem}[theorem]{Construction/Theorem}
    
    \newtheorem{lemma}[theorem]{Lemma}
    \newtheorem{proposition}[theorem]{Proposition}
    
\theoremstyle{definition}
    \newtheorem{remark}[theorem]{Remark}
    
    \newtheorem{example}[theorem]{Example}
    \newtheorem{definition}[theorem]{Definition}
    \newtheorem{construction}[theorem]{Construction}

\newcommand{\comment}[1]{}

\newcommand{\dw}{d_{\circ}}
\newcommand{\db}{d_{\bullet}}


\DeclareMathOperator{\Aut}{Aut}

\title{Combinatorics of pruned Hurwitz numbers}
\author[S.~G.~Fitzgerald]{Sean Gearoid Fitzgerald}
\address{S.~G.~Fitzgerald: School of Mathematics 17, Westland Row, Trinity College Dublin, Dublin 2, Ireland}
\email{sfitzge4@tcd.ie}
\author[M.~A.~Hahn]{Marvin Anas Hahn}
\address{M.~A.~Hahn: School of Mathematics 17, Westland Row, Trinity College Dublin, Dublin 2, Ireland}
\email{hahnma@maths.tcd.ie}
\author[S.~Kelly]{Síofra Kelly}
\address{S.~Kelly: School of Mathematics 17, Westland Row, Trinity College Dublin, Dublin 2, Ireland}
\email{skelly26@tcd.ie}

\begin{document}
\begin{abstract}
    Hurwitz numbers enumerate branched morphisms between Riemannn surfaces with fixed numerical data. They represent important objects in enumerative geometry that are accessible by combinatorial techniques. In the past decade, many variants of Hurwitz numbers have appeared in the literature. In this paper, we focus on an exciting such variant that arises naturally from the theory of topological recursion: Pruned Hurwitz numbers. These are defined as an enumeration of a relevant subset of branched morphisms between Riemann surfaces, that yield smaller numbers than their classical counterparts while retaining maximal information. Thus, pruned Hurwitz numbers may be viewed as the \textit{core} of the Hurwitz problem. In this paper, we develop the combinatorial theory of pruned Hurwitz numbers. In particular, motivated by the successful application of combinatorial techniques to classical Hurwitz numbers, we derive two new combinatorial expressions of pruned Hurwitz numbers. Firstly, we show that they may be expressed in terms of Hurwitz mobiles which are tree-like structure that arise from the theory of random planar maps. Secondly, we prove a tropical correspondence theorem which allows the enumeration of pruned Hurwitz numbers in terms of tropical covers.
\end{abstract}
\maketitle

\section{Introduction}

\subsection{Single and double Hurwitz numbers}
Hurwitz numbers count branched coverings of Riemann surfaces with fixed ramification data and genera. In this work, we are particularly interested in two important families, that arise from specific kinds of ramification data: \textit{single Hurwitz numbers} and \textit{double Hurwitz numbers}.

Single Hurwitz numbers enumerate branched coverings of the projective line with arbitrary ramification of $0$ and simple ramification else.
\begin{definition}
\label{def:hurwitz}
    Let $d>0$, $g\ge0$, $\mu$ a partition of $d$ and let $b=2g-2+\ell(\mu)+d$. We fix $p_1,\dots,p_b\in\mathbb{P}^1$. We call a holomorphic map $f\colon S\to\mathbb{P}^1$ a Hurwitz cover of type $(g,\mu)$ if
    \begin{itemize}
        \item $S$ is a compact, connected Riemann surface of genus $g$,
        \item the ramification profile of $0$ is $\mu$,
        \item the ramification profile of $p_i$ is $(2,1,\dots,1)$,
        \item the pre-images of $0$ are labelled by $1,\dots,\ell(\mu)$, such that the pre-image labelled $i$ has ramification index $\mu_i$.
    \end{itemize}
    We call two Hurwitz covers $f\colon S\to\mathbb{P}^1$ and $f'\colon S'\to\mathbb{P}^1$ equivalent if there exists an isomorphism $g\colon S\to S'$, such that $f=f'\circ g$. We denote by $\mathfrak{H}_g(\mu)$ the set of all equivalence classes of Hurwitz covers of type $(g,\mu,\nu)$.\\
    Finally, we define the associated single Hurwitz number as
    \begin{equation}
        H_g(\mu)=\sum_{[f]\in\mathfrak{H}_g(\mu)}\frac{1}{|\mathrm{Aut}(f)|}.
    \end{equation}
\end{definition}

Dating back to Hurwitz' original work \cite{hurwitz1891riemann}, single Hurwitz numbers are well--studied objects with various remarkable properties. A striking example is the celebrated ELSV formula \cite{ekedahl1999hurwitz,ekedahl2000hurwitz} which relates single Hurwitz numbers to the intersection theory of the moduli space $\overline{\mathcal{M}}_{g,n}$ of genus $g$ curves with $n$ marked points. As a direct corollary one obtains that single Hurwitz numbers are polynomials up to a combinatorial factor.\\
More precisely, for fixed $n$, there exists a polynomial $P_g$ in $n$ variables, such that for all partitions $\mu=(\mu,\dots,\mu_n)$, we have
\begin{equation}
    \frac{1}{b!}H_g(\mu)=\prod_{i=1}^n\frac{\mu_i^{\mu_i}}{\mu_i!}P_g(\mu_1,\dots,\mu_n).
\end{equation}

Hurwitz already made the observation that Hurwitz numbers may be expressed as a combinatorial factorisation problem in the symmetric group. For single Hurwitz numbers this manifests as follows.

\begin{theorem}
Let $g$ be a non-negative integer, $d$ a positive integer and $\mu$ a partition of $d$. Furthermore, let $b=2g-2+\ell(\mu)+d$. Then, we have
 \begin{equation}
     H_g(\mu)= \frac{1}{d!} 
    \left|  \left\{ \begin{array}{l}
            (\sigma, \tau_1, \ldots, \tau_b) \textrm{ such that:} \\
            \bullet \ \sigma,\tau_i \in S_d, \ 1\leq i \leq b, \\
            \bullet \ \mathcal{C} (\sigma) = \mu, \\ 
            \bullet \ \text{the }\tau_i \text{ are transpositions}, \\
            \bullet \ \sigma \tau_1 \cdots \tau_b =\mathrm{id}, \\
            \bullet \ \textrm{the cycles of }\sigma\textrm{ are labelled by }1,\dots,\ell(\mu),\\\textrm{such that the cycle labeleld }i\textrm{ has length }\mu_i,\\
            \bullet \ \text{the subgroup generated by } (\sigma, \tau_1, \ldots, \tau_b)\\
            \text{ acts transitively on } [d].
          
    \end{array}\right\} \right|,
 \end{equation}   
   where $\mathcal{C}(\sigma)$ denotes the cycle type of $\sigma$, i.e. the partition of $d$ corresponding to its conjugacy class.
\end{theorem}

In \cite{zbMATH01539310}, Okounkov introduced a natural generalisation of single Hurwitz numbers, namely \textit{double Hurwitz numbers} and studied them in the context of integrable systems. 

\begin{definition}
\label{def:hurwitzdouble}
    Let $d>0$, $g\ge0$, $\mu,\nu$ partitions of $d$ and let $b=2g-2+\ell(\mu)+\ell(\nu)$. We fix $p_1,\dots,p_b\in\mathbb{P}^1$. We call a holomorphic map $f\colon S\to\mathbb{P}^1$ a Hurwitz cover of type $(g,\mu,\nu)$ if
    \begin{itemize}
        \item $S$ is a compact, connected Riemann surface of genus $g$,
        \item the ramification profile of $0$ is $\mu$, the ramification profile of $\infty$ is $\nu$,
        \item the ramification profile of $p_i$ is $(2,1,\dots,1)$,
        \item the pre-images of $0$ (resp. $\infty$) are labelled by $1,\dots,\ell(\mu)$ (resp. $1,\dots,\ell(\nu)$), such that the pre-image labelled $i$ has ramification index $\mu_i$ (resp. $\nu_j$).

    \end{itemize}
    We call two Hurwitz covers $f\colon S\to\mathbb{P}^1$ and $f'\colon S'\to\mathbb{P}^1$ equivalent if there exists an isomorphism $g\colon S\to S'$, such that $f=f'\circ g$. We denote by $\mathfrak{H}_g(\mu,\nu)$ the set of all equivalence classes of Hurwitz covers of type $(g,\mu,\nu)$.\\
    Finally, we define the associated single Hurwitz number as
    \begin{equation}                    
        H_g(\mu,\nu)=\sum_{[f]\in\mathfrak{H}_g(\mu,\nu)}\frac{1}{|\mathrm{Aut}(f)|}.
    \end{equation}
\end{definition}

We see immediately, that $H_g(\mu)=H_g(\mu,(1^d))$.

Analogously to single Hurwitz numbers, double Hurwitz numbers also admit an expression in terms of factorisations in the symmetric group.

\begin{theorem}
Let $g$ be a non-negative integer, $d$ a positive integer and $\mu,\nu$ partitions of $d$. Furthermore, let $b=2g-2+\ell(\mu)+\ell(\nu)$. Then, we have
 \begin{equation}
     H_g(\mu)= \frac{1}{d!} 
    \left|  \left\{ \begin{array}{l}
            (\sigma_1, \tau_1, \ldots, \tau_b,\sigma_2) \text{ such that:} \\
            \bullet \ \sigma_1, \sigma_2, \tau_i \in S_d, \ 1\leq i \leq b, \\
            \bullet \ \mathcal{C} (\sigma_1) = \mu, \mathcal{C}(\sigma_2) = \nu, \\ 
            \bullet \ \text{the }\tau_i \text{ are transpositions}, \\
            \bullet \ \textrm{the cycles of }\sigma_1\textrm{ are labelled by }1,\dots,\ell(\mu),\\\textrm{such that the cycle labeleld }i\textrm{ has length }\mu_i,\\
            \bullet \ \textrm{the cycles of }\sigma_2\textrm{ are labelled by }1,\dots,\ell(\nu),\\\textrm{such that the cycle labeleld }i\textrm{ has length }\nu_i,\\
            \bullet \ \sigma_1 \tau_1 \cdots \tau_b\sigma_2 =\mathrm{id}, \\
            \bullet \ \text{the subgroup generated by } (\sigma_1, \tau_1, \ldots, \tau_b, \sigma_2)\\
            \ \ \text{ acts transitively on } [d].
          
    \end{array}\right\} \right|.
 \end{equation}   
\end{theorem}

Quite remarkably, double Hurwitz numbers share many features with single Hurwitz numbers. In \cite{goulden2005towards}, the geometry of double Hurwitz numbers was explored towards a connection to the intersection theory of moduli spaces resembling the ELSV formula. As a key feature towards such a connection, a polynomial behaviour of double Hurwitz numbers was identified. More precisely, the authors considered the following set-up.\\
Let $m,n>0$ integers and define the hyperplane
\begin{equation}
    \mathcal{H}_{m,n}=\left\{(\mu,\nu)\in\mathbb{N}^m\times\mathbb{N}^n\mid\sum\mu_i=\sum\nu_j\right\}.
\end{equation}
Further,  define a hyperplane arrangement, the so-called \textit{resonance arrangement}, in $\mathcal{H}_{m,n}$
\begin{equation}
    \mathcal{R}_{m,n}=\left\{\sum_{i\in I}\mu_i-\sum_{j\in J}\nu_j=0\mid I\subset[m],J\subset[n]\right\}.
\end{equation}
We call the hyperplanes in $\mathcal{R}_{m,n}$ \textit{walls} and the connected components of $\mathcal{H}_{m,n}\backslash|\mathcal{R}_{m,n}|$ \textit{chambers} -- where $|\mathcal{R}_{m,n}|$ is the support of $\,\mathcal{R}_{m,n}$, i.e. the union of all hyperplanes.

Now, one may consider the map
\begin{align}
    H_g\colon\mathcal{H}_{m,n}&\to\mathbb{Q}\\
    (\mu,\nu)&\mapsto H_g(\mu,\nu).
\end{align}

\begin{theorem}[{\cite[Theorem 2.1]{goulden2005towards}}]
    The map $H_g$ is piecewise polynomial. More precisely, for each chamber $C$ of $\mathcal{R}_{m,n}$, there exists a polynomial $P_g^C$ in $m+n$ variables of degree $4g-3+m+n$, such that for all $(\mu,\nu)\in C$ we have $H_g(\mu,\nu)=P_g^C(\mu,\nu)$.
\end{theorem}

The polynomial structure of double Hurwitz numbers was further studied in \cite{shadrin2008chamber} in genus $0$ employing intersection theoretic methods. The higher genus case was studied in \cite{cavalieri2011wall} via tropical geometry and in \cite{johnson2015double} using a representation theoretic approach.\\
All these works explored the difference of polynomiality in different chambers. More precisely, they derived \textit{wall-crossing formulae}.

\begin{definition}
Let $C_1$ and $C_2$ be two adjacent chambers in $\mathcal{R}_{m,n}$ seperated by a wall $\delta=\sum_{i\in I}\mu_i-\sum_{j\in J}\nu_j$. We assume that $\delta>0$ in $C_1$ and $\delta<0$ in $C_2$. Furthermore, let $g$ be a non-negative integer. Then, we define the associated wall-crossing as
    \begin{equation}
        WC_\delta^g=P_g^{C_1}-P_g^{C_2}.
    \end{equation}
\end{definition}

It turns out that these wall--crossings may again be expressed in terms of Hurwitz numbers with smaller input data. For the purpose of this manuscript, we only state the genus $0$ result.

\begin{theorem}[{\cite{shadrin2008chamber}}]
    Let $m,n$ be positive integers and $C$ a chamber of $\mathcal{R}_{m,n}$ adjacent to a fixed wall $\delta=\sum_{i\in I}\mu_i-\sum_{j\in J}\nu_j=0$ with $\delta>0$ in $C$. Let $\tilde{C}$ be another chamber neighbouring $C$ and sharing the wall $\delta$ in codimension $1$. Then, we have
    \begin{equation}
        WC_\delta^0(\mu,\nu)=\binom{m+n-2}{|I|+|J|-1}\cdot\delta\cdot H_0(\mu_I,(\nu_J,\delta))\cdot H_0((\mu_{I^c},\delta),(\nu_{J^c})).
    \end{equation}
    for all $\mu,\nu\in C$.
\end{theorem}

We want to highlight the work in \cite{cavalieri2011wall} employing a tropical approach towards Hurwitz numbers. It is based on the derivation of a tropical interpretation of double Hurwitz numbers obtained in \cite{cavalieri2010tropical}, i.e. an expression of double Hurwitz numbers in terms of maps between combinatorial graphs, so-called \textit{tropical covers}. The wall--crossing formulae in arbitrary genus, then follow from an intricate analysis of spaces of these graphs in different chambers of polynomiality. In particular, the Hurwitz numbers with smaller input data in the wall--crossing formulae arise from cutting the involved graphs into smaller graphs, each contributing to a smaller Hurwitz problem.\vspace{\baselineskip}

A non-tropical, combinatorial study of double Hurwitz numbers in genus $0$ was undertaken in \cite{duchi2014bijections}. The basis of this work, is a well--known graph-theoretic interpretation of double Hurwitz numbers. We note that while also an expression in terms of graphs, this interpretation differs significantly from the previously mentioned tropical one. Namely, while the tropical correspondence involves a \textit{weighted bijection}, the one employed in \textit{loc.cit.} is actually a $1$-to-$1$ bijection. The involved graphs are called \textit{Hurwitz galaxies}. Starting from Hurwitz galaxies, the authors of \cite{duchi2014bijections}, proceed to employ an ingenious idea from graph theory, similar to Schaeffer's bijection, to derive a correspondence between Hurwitz galaxies and tree-like structure called \textit{Hurwitz mobiles}. Due to their close proximity to trees, Hurwitz mobiles allow for simplified counting arguments, which in \textit{loc.cit. } enabled a new proof of Hurwitz' original closed formula for genus $0$ single Hurwitz numbers and the resolution of a conjecture of Kazarian and Zvonkine.

\subsection{Pruned Hurwitz numbers}
In the past decade, many new variants of Hurwitz numbers were introduced. In this work, we focus on so-called \textit{pruned single Hurwitz numbers}, which we will denote $PH_g(\mu)$, that were first studied in \cite{do2018pruned} in the context of topological recursion.\footnote{We note that the main ideas were already present in \cite{zvonkine2004algebra,irving2009minimal}.} Pruned Hurwitz numbers were originally defined as an enumeration of a subset of the branched covers enumerated by single Hurwitz numbers (see \cref{def:pruned} for a precise definition). We denote, for a fixed partition $\mu$ and fixed genus $g$, pruned single Hurwitz numbers by $PH_g(\mu)$ (for a precise definition, see \cref{def:pruned}). This enumeration has remarkable combinatorial properties and pruned single Hurwitz numbers are in many ways better behaved than their classical counterparts. For example, for fixed $n$, there exists a polynomial $Q_g$ in $n$ variables, such that

\begin{equation}
    \frac{1}{b!}PH_g(\mu)=Q_q(\mu_1,\dots,\mu_n),
\end{equation}

i.e. pruned single Hurwitz numbers are actually polynomial without an involved combinatorial pre-factor \cite[Theorem 1]{do2018pruned}.  Moreover, it was proved in \cite{do2018pruned} that pruned single Hurwitz determine their classical counterparts and vice versa. Therefore, they provide a smaller, in some sense better behaved, enumerative invariant that may be viewed as the \textit{core} of the Hurwitz number problem.\\
In \cite{zbMATH06791415}, pruned double Hurwitz numbers $PH_g(\mu,\nu)$ were introduced. Analogously to the single case, pruned double Hurwitz numbers enumerate a subset of $\mathfrak{H}_g(\mu,\nu)$ and again this smaller enumerative invariant captures the key features of double Hurwitz numbers. In particular, pruned double Hurwitz numbers are piecewise polynomial with the same chamber structure as their classical counterparts. Wall--crossing formulae for pruned double Hurwitz numbers however remain an open problem. Moreover, mirroring the idea for single pruned Hurwitz numbers, it was proved in \cite{zbMATH06791415} that pruned double Hurwitz numbers determine double Hurwitz numbers and vice versa.

\subsection{Main results}
The aim of this paper is to explore the combinatorics of pruned Hurwitz numbers and lay the foundations for further combinatorial analysis.\\
The first part is dedicated to applying the techniques from \cite{duchi2014bijections} to this setting. In \cite{hahn2020bi}, an interpretation of pruned single and double Hurwitz numbers in terms of Hurwitz galaxies was given. Due to their aptitude for enumerative applications, it is natural to ask for an expression of pruned Hurwitz numbers in terms of Hurwitz mobiles. In order to achieve this, we review the bijection between Hurwitz galaxies and Hurwitz mobiles derived in \cite{duchi2014bijections} in \cref{sec:bij}. This bijection proceeds through several steps, by enhancing and altering Hurwitz galaxies. While the combinatorial data is preserved through this process, it becomes reencoded. For our purposes, we rephrase the bijection in the rich language of \textit{Dyck paths}. The properties and applications of Dyck paths and their various forms (also known as contour functions, Dyck codes, standard contour codes etc.) have been the subject of much combinatorial research. \cite{dyckmodk,dycknondecreasing,dycknondecreasingenum,dyckcoloured} Our rephrasing of the bijection in \cite{duchi2014bijections} in this language, allows us to perform an intricate combinatorial analysis to derive a new correspondence theorem expressing pruned single Hurwitz numbers in terms of what we call \textit{pruned Hurwitz mobiles} in \cref{sec:bij}. This achieves an enumeration of genus $0$ pruned single Hurwitz numbers in terms of tree-like graphs.\\
In the second part, we study the tropical combinatorics of pruned double Hurwitz numbers. While in the past, tropical derivations of enumerative invariants were obtained either using representation theoretic methods \cite{cavalieri2018graphical,cavalieri2010tropical} or via degeneration techniques applied to the underlying algebro-geometric objects \cite{cavalieri2016tropicalizing}, we employ a purely combinatorial approach. The starting point is a cut--and--join recursion of pruned double Hurwitz numbers derived in \cite{zbMATH06791415}. By tracing this recursion back to its infinite set of base cases, we build tropical covers without the use of any underlying geometry or representation theory in \cref{sec:troppru}. The correspondence theorem, expressing pruned double Hurwitz numbers in terms of tropical covers is then obtained by analysing the multiplicity of each cover in the recursion. This allows us to study the polynomiality of pruned double Hurwitz numbers via tropical techniques in \cref{sec:polytrop}. 

\subsection{Structure}
In \cref{sec:graphs}, we review the basics on branching graphs, Hurwitz galaxies, Hurwitz mobiles and Dyck paths. These are the main objects in the bijection correspondences between $\mathfrak{H}_g(\mu,\nu)$ and graph theoretic structures. Based on this, we introduce pruned single and double Hurwitz numbers in \cref{sec:prunhur} and discuss some of their key properties in more detail. In \cref{sec:bij}, we sketch the bijection between covers contributing to genus zero double Hurwitz numbers and Hurwitz mobiles obtained in \cite{duchi2014bijections}. We rephrase this bijection in terms of Dyck paths to prove our correspondence theorem enumerate genus zero pruned single Hurwitz numbers in terms of pruned Hurwitz mobiles. After that, we move on to the tropical combinatorics of pruned double Hurwitz numbers and first introduce some basics on tropical covers in \cref{sec:trop}. We prove the correspondence theorem, enumerating pruned double Hurwitz numbers via weighted tropical covers in \cref{sec:troppru}. Finally, we explore the polynomiality of pruned double Hurwitz numbers via their new tropical interpretation in \cref{sec:polytrop}.

\subsection*{Acknowledgements} The first and third author acknowledge partial support by the Hamilton Trust fund during the work on this paper.

\section{Bijections for Hurwitz numbers: Branching graphs, Galaxies and mobiles}
\label{sec:graphs}
In this section, we review a classical construction expressing double Hurwitz numbers in terms of graphs on surfaces and explore some of their combinatorics. There are several equivalent ways to produce such a correspondence. 

\subsection{Branching graphs and Galaxies}
First, we focus on two of them: \textit{Branching graphs} and \textit{Hurwitz galaxies}. These graphs are obtained by pulling back graphs on $\mathbb{P}^1$ along Hurwitz covers.
We begin with the following definition.

\begin{definition} \label{def:goodgraph}
A good graph on a surface $S$ is a graph $\Gamma$ embedded on $S$ such that:
\begin{enumerate}
    \item $S \backslash \Gamma$ is homeomorphic to a disjoint union of open disks,
    \item Wherever two edges cross there is a vertex,
    \item Edges that end without a vertex are called half-edges.
\end{enumerate}
\end{definition}

\begin{definition}
Let $d$ be a positive integer, $g$ a non-negative integer and $\mu, \nu$ be ordered partitions of $d$. Again, let $b=2g-2+\ell(\mu)+\ell(\nu)$. We define a branching graph of type $(g, \mu, \nu)$ to be a good graph $\Gamma$ on a compact oriented surface $S$ of genus $g$ that satisfies the following:
\begin{enumerate}
    \item There are $\ell(\nu)$ vertices, labelled $1, \dots, \ell(\nu)$. Each vertex is adjacent to $\nu_i \cdot b$ half-edges, labelled cyclically counter-clockwise by $1, \dots, b$. 
    \item There are exactly $b$ full edges labelled by $1, \dots, b$.
    \item The $\ell(\mu)$ faces are labelled by $1, \dots, \ell(\mu)$ and the face labelled $i$ has perimeter which we denote by $per(i)=\mu_i$.
\end{enumerate}
We define an isomorphism of branching graphs as an orientation-preserving homeomorphism of surfaces which induces an isomorphism of graphs preserving vertex, edge and face labels. Further, we denote by $B_g(\mu,\nu)$ the set of all isomorphism classes of branching graphs of type $(g,\mu,\nu)$.
\end{definition}

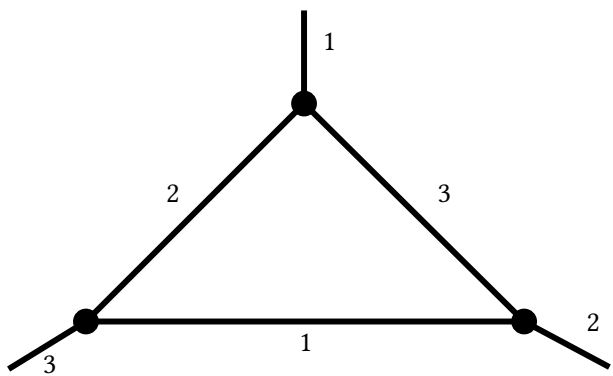
\begin{figure}
\tikzset{every picture/.style={line width=0.75pt}} 

\begin{tikzpicture}[x=0.75pt,y=0.75pt,yscale=-1,xscale=1]

\draw  [fill={rgb, 255:red, 0; green, 0; blue, 0 }  ,fill opacity=1 ][line width=2.25]  (207,194) .. controls (207,191.24) and (209.24,189) .. (212,189) .. controls (214.76,189) and (217,191.24) .. (217,194) .. controls (217,196.76) and (214.76,199) .. (212,199) .. controls (209.24,199) and (207,196.76) .. (207,194) -- cycle ;
\draw  [fill={rgb, 255:red, 0; green, 0; blue, 0 }  ,fill opacity=1 ][line width=2.25]  (317,84) .. controls (317,81.24) and (319.24,79) .. (322,79) .. controls (324.76,79) and (327,81.24) .. (327,84) .. controls (327,86.76) and (324.76,89) .. (322,89) .. controls (319.24,89) and (317,86.76) .. (317,84) -- cycle ;
\draw  [fill={rgb, 255:red, 0; green, 0; blue, 0 }  ,fill opacity=1 ][line width=2.25]  (428,194) .. controls (428,191.24) and (430.24,189) .. (433,189) .. controls (435.76,189) and (438,191.24) .. (438,194) .. controls (438,196.76) and (435.76,199) .. (433,199) .. controls (430.24,199) and (428,196.76) .. (428,194) -- cycle ;
\draw [line width=2.25]    (322,84) -- (212,194) ;
\draw [line width=2.25]    (322,84) -- (433,194) ;
\draw [line width=2.25]    (212,194) -- (433,194) ;
\draw [line width=2.25]    (322,37) -- (322,84) ;
\draw [line width=2.25]    (212,194) -- (173,218) ;
\draw [line width=2.25]    (433,194) -- (476,217) ;

\draw (251,123.4) node [anchor=north west][inner sep=0.75pt]    {$2$};
\draw (318,198.4) node [anchor=north west][inner sep=0.75pt]    {$1$};
\draw (388,123.4) node [anchor=north west][inner sep=0.75pt]    {$3$};
\draw (330,46.4) node [anchor=north west][inner sep=0.75pt]    {$1$};
\draw (463,188.4) node [anchor=north west][inner sep=0.75pt]    {$2$};
\draw (189,209.4) node [anchor=north west][inner sep=0.75pt]    {$3$};

\end{tikzpicture}
\caption{A branching graph of type $(0,(2,1),(1,1,1)$.}
\label{fig:brgraph}
\end{figure}

We illustrate a branching graph in the following example.

\begin{example}
\label{ex:brgraph}
We fix $g=0$, $\mu=(2,1)$ and $\nu=(1,1,1)$. A branching graph of type $(0,\mu,\nu)$ is depicted in \cref{fig:brgraph}. All three vertices have valency $3$. Since $b=3$, this reflects the partition $\nu$. The outer face has perimeter $2$, while the inner face has perimeter $1$ corresponding to $\mu$.
\end{example}

The following folklore theorem connects branching graphs to double Hurwitz numbers

\begin{theorem}
\label{thm-brangr}
    Let $g$ be a non-negative integer, $d$ a positive integer and $\mu,\nu$ partitions of $d$. Then, we have
    \begin{equation}
        H_g(\mu,\nu)=\sum_{[\Gamma]\in B_g(\mu,\nu)}\frac{1}{|\mathrm{Aut}(\Gamma)|}.
    \end{equation}
\end{theorem}

\begin{proof}[Sketch of proof]
Let $\zeta_1,\dots,\zeta_b$ be the $b-$th roots of unity. The idea behind the proof is to consider the \textit{star graph} $G$ on $\mathbb{P}^1$ with a unique vertex $v$ at $\infty$ and non-intersecting half-edges connecting $v$ to each $\zeta_i$ (see \cref{fig:stargraph}). Choosing $p_i=\zeta_i$ in \cref{def:hurwitzdouble} and a Hurwitz cover $f\colon S\to\mathbb{P}^1$ of type $(g,\mu,\nu)$, we may consider the graph $\Gamma=f^{-1}(G)\subset S$, where vertices of $\Gamma$ are pre-images of the vertex on $\mathbb{P}^1$ and edges are obtained by two half-edges meeting at a pre-image of $\zeta_i$. It is easy to see that $f^{-1}(G)$ is indeed a branching graph of type $(g,\mu,\nu)$. For the other direction, we observe that a branching graph is a good graph with half-edges. Thus, the Riemann existence theorem together with the Galois correspondence for topological covers allows to re-construct the original branched covering of $\mathbb{P}^1$.
\end{proof}

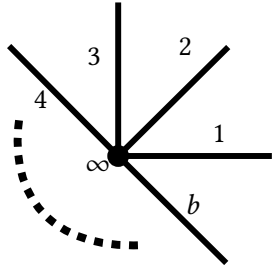
\begin{figure}

\tikzset{every picture/.style={line width=0.75pt}} 

\begin{tikzpicture}[x=0.75pt,y=0.75pt,yscale=-1,xscale=1]

\draw  [fill={rgb, 255:red, 0; green, 0; blue, 0 }  ,fill opacity=1 ][line width=2.25]  (326,137) .. controls (326,134.79) and (327.79,133) .. (330,133) .. controls (332.21,133) and (334,134.79) .. (334,137) .. controls (334,139.21) and (332.21,141) .. (330,141) .. controls (327.79,141) and (326,139.21) .. (326,137) -- cycle ;
\draw [line width=2.25]    (407.5,137) -- (330,137) ;
\draw [line width=2.25]    (385.5,82) -- (330,137) ;
\draw [line width=2.25]    (330,59.5) -- (330,137) ;
\draw [line width=2.25]    (275,81.5) -- (330,137) ;
\draw [line width=2.25]    (330,137) -- (384.5,191) ;
\draw [line width=3]  [dash pattern={on 3.38pt off 3.27pt}]  (280,119) .. controls (275,157) and (296,184) .. (342,182) ;

\draw (312,137.4) node [anchor=north west][inner sep=0.75pt]    {$\infty $};
\draw (376,119.4) node [anchor=north west][inner sep=0.75pt]    {$1$};
\draw (359,75.4) node [anchor=north west][inner sep=0.75pt]    {$2$};
\draw (313,80.4) node [anchor=north west][inner sep=0.75pt]    {$3$};
\draw (286,102.4) node [anchor=north west][inner sep=0.75pt]    {$4$};
\draw (363,154.4) node [anchor=north west][inner sep=0.75pt]    {$b$};

\end{tikzpicture}

    \caption{The star graph on $\mathbb{P}^1$.}
    \label{fig:stargraph}
\end{figure}

Next, we define Hurwitz galaxies.

\begin{definition}
 Let $d$ be a positive integer, $g$ a non-negative integer, $\mu,\nu$ partitions of $d$, and $b=2g-2+\ell(\mu)+\ell(\nu)$. Then, we define a Hurwitz galaxy of type $(g,\mu,\nu)$ to be a graph $G$, such that
\begin{enumerate}
    \item $G$ partitions $S$ into $\ell(\mu)+\ell(\nu)$ disjoint faces homeomorphic to an open disk.
    \item The faces are bi-coloured black and white, such that each edge is adjacent to a white and a black face.
    \item The white faces are labelled by $1, \dots, l(\mu)$, the black faces by $1, \dots, l(\nu)$, such that the boundary of each white face labelled $i$ contains $\mu_i \cdot (b+1)$ vertices and the boundary of each black face $i$ contains $\nu_i \cdot (b+1)$ vertices. 
    \item The vertices are coloured cyclically clockwise with respect to the adjacent white faces by $0,1, \dots, b$.
    \item For each $i \in \{1, \dots, b\}$ there are $d-1$ vertices with colour $i$, where $d-2$ of these vertices are 2-valent, and one is 4-valent. There are $d$ vertices with colour $0$, each of which are 2-valent. 
\end{enumerate}
An isomorphism between Hurwitz galaxies is an orientation-preserving homeomorphism of their respective surfaces which induces an isomorphism of graphs preserving vertex, edge and face labels. We denote by $G_{g}(\mu,\nu)$ be the set of isomorphism classes of Hurwitz galaxies of type $(g,\mu,\nu)$.
\end{definition}

\begin{example}
We fix $g,\mu,\nu$ as in \cref{ex:brgraph}. An example of a Hurwitz galaxy of type $(0,(2,1),(1,1,1))$ is depicted in \cref{fig:hurgal}. The arrows indicate the orientation of the faces with respect to the ordering of the adjacent vertices.
\begin{figure}
    \centering
    \includegraphics[scale=0.6]{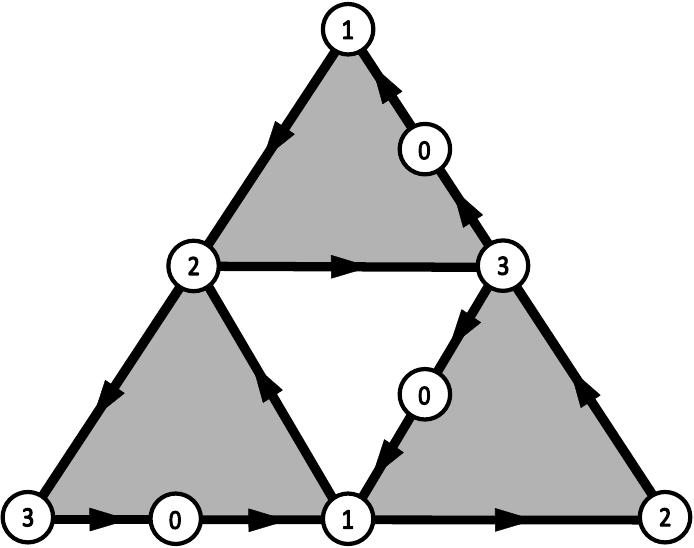}
    \caption{A Hurwitz galaxy of type $(0,(2,1),(1,1,1))$.}
    \label{fig:hurgal}
\end{figure}
\end{example}

\begin{figure}
    \centering
    \includegraphics[scale=0.6]{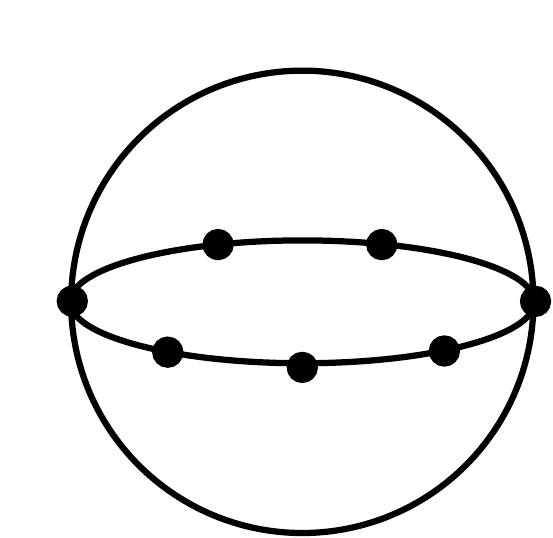}
    \caption{The circle graph on the Riemann sphere whose vertices are $\zeta_1,\dots,\zeta_b$, the $b$-th roots of unity for fixed $b$.}
    \label{fig:circlegraph}
\end{figure}

The following theorem is proved similarly to \cref{thm-brangr} by pulling back the graph in \cref{fig:circlegraph} whose vertices are the $b$-th roots of unity.

\begin{theorem}
Let $g$ be a non-negative integer, $d$ a positive integer and $\mu,\nu$ partitions of $d$. Then, we define
\begin{equation}
    H_g(\mu,\nu)=\sum_{[G]\in G_g(\mu,\nu)}\frac{1}{|\mathrm{Aut}(G)|}.
\end{equation}
\end{theorem}

\begin{remark}
\label{rem:brgalcorr}
 As a consequence of the above discussion, there is a bijection between branching graphs and Hurwitz galaxies. The combinatorial correspondence is described in \cite[Figure 3, Proposition 9]{zbMATH06791415}. In fact, the branching graph in \cref{fig:brgraph} and the Hurwitz galaxy in \cref{fig:hurgal} correspond to each other and give rise to the same branched cover in $\mathfrak{H}_g(\mu,\nu)$.
\end{remark}

\subsection{Distance Labelling and Geodesic Edges}
We now explore some of the combinatorics of Hurwitz galaxies. For a more detailed account, we refer to \cite{duchi2014bijections}. To begin with, we define a marked Hurwitz galaxy as a tuple $(G,x_0)$ where $G$ is a Hurwitz galaxy as defined above, and $x_0$ is a vertex of $G$ with colour $0$. When the marked vertex is clear from the context, we mostly denote a marked Hurwitz galaxy $(G,x_0)$ by $G$. We denote by $\mathcal{G}_g(\mu,\nu)$ the set of isomorphism classes of marked Hurwitz galaxies of type $(g,\mu,\nu)$, where isomorphisms respect the marked vertex.\\
The introduction of this distinguished vertex allows us to define a notion of distance on a Hurwitz galaxy.

\begin{definition}
    Let $G$ be a Hurwitz galaxy with marked vertex $x_0$. We define the distance labelling of a vertex $x$ of $G$ to be the number $\delta(x)$ of edges contained in a shortest oriented path from $x_0$ to $x$.
\end{definition}

Note that every vertex can be reached by an oriented path from $x_0$ due to the fact that $G$ is connected and each oriented edge belongs to a cycle. Thus $\delta$ is well-defined and we get the following properties:
\begin{enumerate}
    \item the colour and distance label of a vertex $x$ are related by $c(x)\equiv \delta (x)\, \mathrm{mod}\, b+1$. 
    \item For any oriented edge $x \rightarrow y$, $\delta(y) \equiv \delta(x)+1$ (mod $b+1$). Since the distance label of $y$ corresponds to the length of the shortest oriented path to $y$, we also must have $\delta(y) \leq \delta(x) +1$. 
\end{enumerate}

\begin{definition}
    Let $G$ be a marked Hurwitz galaxy with distinguished vertex $x_0$. We define the weight of an edge $e=x \rightarrow y$ as the non-negative integer quantity given by:
\begin{equation}\label{eq:weightdef}
    w(e)=\frac{\delta(x)+1-\delta(y)}{b+1}
\end{equation}
An edge $e$ is called geodesic if $w(e)=0$ and non-geodesic otherwise.  
\end{definition}

\begin{example}
We consider the Hurwitz galaxy in \cref{fig:hurgal} and fix the $0$ colored vertex on the bottom left as the marked vertex. The resulting marked Hurwitz galaxy is depicted in \cref{fig:hurgalmarked}.  The marked vertex is coloured green and the non-geodesic edges are coloured yellow. The vertices are labelled by $(a,b)$, where $a$ is the colour of the vertex and $b$ is the distance from the marked vertex.

\begin{figure}
    \centering
    \includegraphics[scale=0.7]{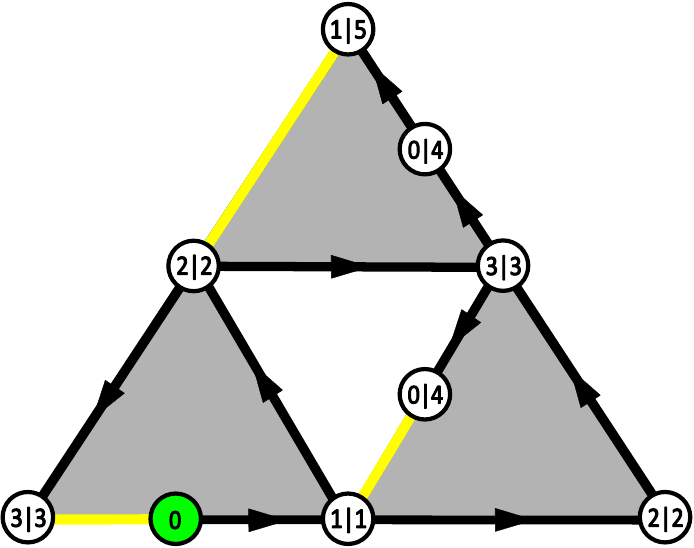}
    \caption{A marked Hurwitz galaxy of type $(0,(2,1),(1,1,1))$. The marked vertex is coloured green and the non-geodesic edges are coloured yellow. The vertices are labelled by $(a,b)$, where $a$ is the colour of the vertex and $b$ is the distance from the marked vertex.}
    \label{fig:hurgalmarked}
\end{figure}
\end{example}

Consider some arbitrary vertex $v$ of a face $F$ of degree $i$ of a marked Hurwitz galaxy $G$. Let $p$ be some path around the boundary of $F$ that starts at $v$, passes through each edge of $F$ once and finishes at $v$ (a total of $(b+1)i$ edges). Then the change in distance label from the beginning of $p$ to the end of $p$ is given by
\begin{equation}
    \sum_{e=[x,y] \in p} \delta(x)-\delta(y).
\end{equation}
Since $p$ begins and ends at $p$ however, the total variation in distance labels is in fact equal to $0$. Thus:
\begin{equation}
    \sum_{e=[x,y] \in p} \delta(x)-\delta(y) =0
\end{equation}
Using the definition of $w(e)$ as in equation (\ref{eq:weightdef}) we conclude that:
\begin{equation}
    \sum_{e=[x,y] \in p} w(e)=  
    \sum_{e=[x,y] \in p} \frac{\delta(x)+1-\delta(y)}{b+1}
    = \sum_{e=[x,y] \in p} \frac{1}{b+1} = \frac{(b+1)i}{b+1}
    =i
\end{equation}
Thus, we have the following lemma.

\begin{lemma}
    Let $G$ a marked Hurwitz galaxy. Then, the sum of the weight of the edges incident to any face with degree $i$ is $i$.
\end{lemma}

\subsection{Hurwitz mobiles}
Hurwitz mobiles are tree-like structures introduced in \cite{duchi2014bijections} consisting of black polygons, white polygons and edges between them. They are in some sense easier to work with than Hurwitz galaxies or branching graphs, due to their tree-like behaviour.\\
Since our use of Hurwitz mobiles is restricted to the genus $0$ case, we also only give the definition in that situation. In \cite{duchi2014bijections}, the Hurwitz mobiles we consider are called \textit{free}.

\begin{definition}
    Let $d$ be a positive integer and $\mu, \nu$ be ordered partitions of $d$. Furthermore let $b=\ell(\mu)+\ell(\nu)-2$. A Hurwitz mobile of type $(\mu, \nu)$ is then a connected partially oriented graph consisting of:  
\begin{enumerate}
    \item $d$ white nodes forming $\ell(\mu)$ disjoint oriented simple cycles, of length $\mu_i$, for $i=1, \dots ,\ell(\mu)$. We refer to these cycles as white polygons.
    \item $d$ black nodes forming $\ell(\nu)$ disjoint oriented simple cycles, of length $\nu_i$, for $i=1, \dots ,\ell(\nu)$. We refer to these cycles as black polygons.
    \item $b$ non-oriented edges with non-negative weights such that:
    \begin{itemize}
        \item each zero weight edge has its endpoints on white polygons.
        \item each positive weight edge is incident to a black polygon and a white polygon.
        \item The sum of the weights of edges incident to any $i$-gon is $i$.
    \end{itemize}
\end{enumerate}
An edge-labelled Hurwitz mobile is then a Hurwitz mobile in which each of the weighted edges has a distinct associated label taken in the set $\{ 0, 1, \dots , b \}$. We denote the set of Hurwitz mobiles of type $(\mu,\nu)$ by $M(\mu,\nu)$.
\end{definition}

\begin{example}
Let $\mu=(2,1)$ and $\nu=(1,1,1)$. A Hurwitz mobile of type $((2,1),(1,1,1))$ is depicted in \cref{fig:hurmobile}. The label of the edges is depicted as numbers inside circles, whereas the number of dashes of each edge indicates their weight. For example, the edge labelled $3$ has weight $0$, whereas the edge labelled $0$ has weight $1$.

\begin{figure}
    \centering
    \includegraphics{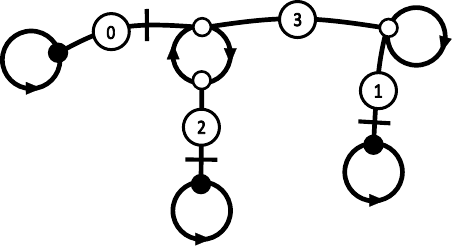}
    \caption{A Hurwitz mobile of type $((2,1),(1,1,1))$}
    \label{fig:hurmobile}
\end{figure}
\end{example}

Given a mobile $M$, there exists an operation $\sigma$ on $M$ called a shift that when applied to $M$, yields another mobile $\sigma(M)$. This operation is of interest to us in this report as it partitions $\mathcal{M}(\mu,\nu)$, the set of Hurwitz mobiles, into so called shift-equivalence classes.

\begin{definition}
Let $M$ be a Hurwitz mobile. Its shift $\sigma(M)$ is the Hurwitz mobile obtained by translating the two endpoints of the edge of label $b$ to the next vertices of the adjacent polygons according to their orientation, and then incrementing each edge label in $M$ by $1$ modulo $b+1$.\\
We call two Hurwitz mobiles shift equivalent if  one may be obtained by the other by a finite sequence of shifts. We will denote by $\mathcal{M}(\mu,\nu)$ the set of shift equivalence classes of mobiles in $M(\mu,\nu)$.
\end{definition}

The local structure of the shift is depicted in \cref{fig:shiftex}

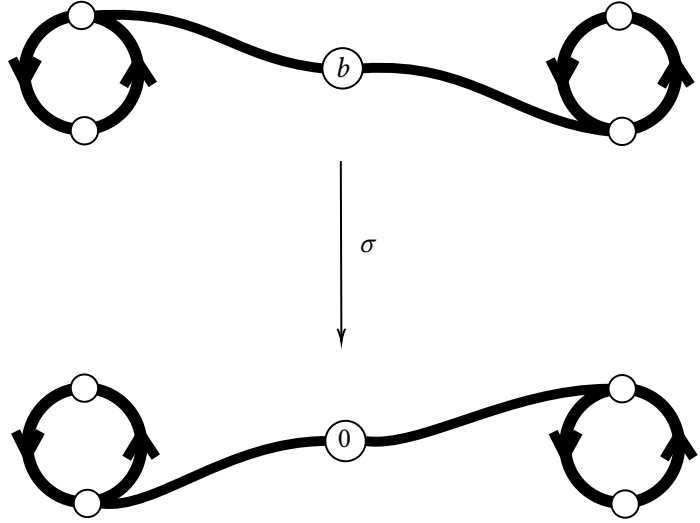
\begin{figure}[H]

\tikzset{every picture/.style={line width=0.75pt}} 

\begin{tikzpicture}[scale=0.6,x=0.75pt,y=0.75pt,yscale=-1,xscale=1]

\draw  [line width=4.5]  (44,82.5) .. controls (44,55.99) and (65.49,34.5) .. (92,34.5) .. controls (118.51,34.5) and (140,55.99) .. (140,82.5) .. controls (140,109.01) and (118.51,130.5) .. (92,130.5) .. controls (65.49,130.5) and (44,109.01) .. (44,82.5) -- cycle ;
\draw [line width=3.75]    (34,71) -- (44,91) ;
\draw [line width=3.75]    (44,91) -- (55,71) ;
\draw [line width=3.75]    (139.5,73.5) -- (127.5,93.5) ;
\draw [line width=3.75]    (139.5,73.5) -- (152,92.5) ;
\draw  [fill={rgb, 255:red, 255; green, 255; blue, 255 }  ,fill opacity=1 ] (80.5,34.5) .. controls (80.5,28.15) and (85.65,23) .. (92,23) .. controls (98.35,23) and (103.5,28.15) .. (103.5,34.5) .. controls (103.5,40.85) and (98.35,46) .. (92,46) .. controls (85.65,46) and (80.5,40.85) .. (80.5,34.5) -- cycle ;
\draw  [fill={rgb, 255:red, 255; green, 255; blue, 255 }  ,fill opacity=1 ] (83,131.5) .. controls (83,125.15) and (88.15,120) .. (94.5,120) .. controls (100.85,120) and (106,125.15) .. (106,131.5) .. controls (106,137.85) and (100.85,143) .. (94.5,143) .. controls (88.15,143) and (83,137.85) .. (83,131.5) -- cycle ;
\draw  [line width=4.5]  (496,83) .. controls (496,56.49) and (517.49,35) .. (544,35) .. controls (570.51,35) and (592,56.49) .. (592,83) .. controls (592,109.51) and (570.51,131) .. (544,131) .. controls (517.49,131) and (496,109.51) .. (496,83) -- cycle ;
\draw [line width=3.75]    (486,71.5) -- (496,91.5) ;
\draw [line width=3.75]    (496,91.5) -- (507,71.5) ;
\draw [line width=3.75]    (591.5,74) -- (579.5,94) ;
\draw [line width=3.75]    (591.5,74) -- (604,93) ;
\draw  [fill={rgb, 255:red, 255; green, 255; blue, 255 }  ,fill opacity=1 ] (532.5,35) .. controls (532.5,28.65) and (537.65,23.5) .. (544,23.5) .. controls (550.35,23.5) and (555.5,28.65) .. (555.5,35) .. controls (555.5,41.35) and (550.35,46.5) .. (544,46.5) .. controls (537.65,46.5) and (532.5,41.35) .. (532.5,35) -- cycle ;
\draw  [fill={rgb, 255:red, 255; green, 255; blue, 255 }  ,fill opacity=1 ] (535,132) .. controls (535,125.65) and (540.15,120.5) .. (546.5,120.5) .. controls (552.85,120.5) and (558,125.65) .. (558,132) .. controls (558,138.35) and (552.85,143.5) .. (546.5,143.5) .. controls (540.15,143.5) and (535,138.35) .. (535,132) -- cycle ;
\draw   (294.5,79) .. controls (294.5,69.61) and (302.11,62) .. (311.5,62) .. controls (320.89,62) and (328.5,69.61) .. (328.5,79) .. controls (328.5,88.39) and (320.89,96) .. (311.5,96) .. controls (302.11,96) and (294.5,88.39) .. (294.5,79) -- cycle ;
\draw [line width=3.75]    (103.5,34.5) .. controls (196,28.5) and (216,70.5) .. (294.5,79) ;
\draw [line width=3.75]    (328.5,79) .. controls (421,73) and (456.5,123.5) .. (535,132) ;
\draw  [line width=4.5]  (46.5,396) .. controls (46.5,369.49) and (67.99,348) .. (94.5,348) .. controls (121.01,348) and (142.5,369.49) .. (142.5,396) .. controls (142.5,422.51) and (121.01,444) .. (94.5,444) .. controls (67.99,444) and (46.5,422.51) .. (46.5,396) -- cycle ;
\draw [line width=3.75]    (36.5,384.5) -- (46.5,404.5) ;
\draw [line width=3.75]    (46.5,404.5) -- (57.5,384.5) ;
\draw [line width=3.75]    (142,387) -- (130,407) ;
\draw [line width=3.75]    (142,387) -- (154.5,406) ;
\draw  [fill={rgb, 255:red, 255; green, 255; blue, 255 }  ,fill opacity=1 ] (83,348) .. controls (83,341.65) and (88.15,336.5) .. (94.5,336.5) .. controls (100.85,336.5) and (106,341.65) .. (106,348) .. controls (106,354.35) and (100.85,359.5) .. (94.5,359.5) .. controls (88.15,359.5) and (83,354.35) .. (83,348) -- cycle ;
\draw  [fill={rgb, 255:red, 255; green, 255; blue, 255 }  ,fill opacity=1 ] (85.5,445) .. controls (85.5,438.65) and (90.65,433.5) .. (97,433.5) .. controls (103.35,433.5) and (108.5,438.65) .. (108.5,445) .. controls (108.5,451.35) and (103.35,456.5) .. (97,456.5) .. controls (90.65,456.5) and (85.5,451.35) .. (85.5,445) -- cycle ;
\draw  [line width=4.5]  (498.5,396.5) .. controls (498.5,369.99) and (519.99,348.5) .. (546.5,348.5) .. controls (573.01,348.5) and (594.5,369.99) .. (594.5,396.5) .. controls (594.5,423.01) and (573.01,444.5) .. (546.5,444.5) .. controls (519.99,444.5) and (498.5,423.01) .. (498.5,396.5) -- cycle ;
\draw [line width=3.75]    (488.5,385) -- (498.5,405) ;
\draw [line width=3.75]    (498.5,405) -- (509.5,385) ;
\draw [line width=3.75]    (594,387.5) -- (582,407.5) ;
\draw [line width=3.75]    (594,387.5) -- (606.5,406.5) ;
\draw  [fill={rgb, 255:red, 255; green, 255; blue, 255 }  ,fill opacity=1 ] (535,348.5) .. controls (535,342.15) and (540.15,337) .. (546.5,337) .. controls (552.85,337) and (558,342.15) .. (558,348.5) .. controls (558,354.85) and (552.85,360) .. (546.5,360) .. controls (540.15,360) and (535,354.85) .. (535,348.5) -- cycle ;
\draw  [fill={rgb, 255:red, 255; green, 255; blue, 255 }  ,fill opacity=1 ] (537.5,445.5) .. controls (537.5,439.15) and (542.65,434) .. (549,434) .. controls (555.35,434) and (560.5,439.15) .. (560.5,445.5) .. controls (560.5,451.85) and (555.35,457) .. (549,457) .. controls (542.65,457) and (537.5,451.85) .. (537.5,445.5) -- cycle ;
\draw   (297,392.5) .. controls (297,383.11) and (304.61,375.5) .. (314,375.5) .. controls (323.39,375.5) and (331,383.11) .. (331,392.5) .. controls (331,401.89) and (323.39,409.5) .. (314,409.5) .. controls (304.61,409.5) and (297,401.89) .. (297,392.5) -- cycle ;
\draw    (310,157) -- (310.49,307) ;
\draw [shift={(310.5,309)}, rotate = 269.81] [color={rgb, 255:red, 0; green, 0; blue, 0 }  ][line width=0.75]    (10.93,-3.29) .. controls (6.95,-1.4) and (3.31,-0.3) .. (0,0) .. controls (3.31,0.3) and (6.95,1.4) .. (10.93,3.29)   ;
\draw [line width=3.75]    (108.5,445) .. controls (153.5,451.5) and (214.5,393) .. (297,392.5) ;
\draw [line width=3.75]    (331,392.5) .. controls (376,399) and (452.5,349) .. (535,348.5) ;

\draw (304,66.4) node [anchor=north west][inner sep=0.75pt]    {$b$};
\draw (305,380.1) node [anchor=north west][inner sep=0.75pt]    {$0$};
\draw (323.5,220.4) node [anchor=north west][inner sep=0.75pt]    {$\sigma $};

\end{tikzpicture}

  \caption{The action of $\sigma$ on a non-weighted edge of label $r$.}\label{fig:shiftex}
\end{figure}

\begin{proposition}
    Given a Hurwitz mobile $M\in M(\mu,\nu)$, then we have $\sigma^{b+1}(M)=M$ and there are $b+1$ distinct graphs in the shift equivalence class of each mobile in $M(\mu,\nu)$.
\end{proposition}

\begin{example}
    The shift equivalence class of the Hurwitz galaxy in \cref{fig:hurgalmarked} is depicted in \cref{fig:shiftequivalence}.

    \begin{figure}[H] 
  \centering
 \begin{subfigure}[b]{0.3\linewidth}
  \centering
  \includegraphics[width=0.9\linewidth]{mobileexample.pdf}
  \caption{}
 \end{subfigure}
    \begin{subfigure}[b]{0.3\linewidth}
  \centering
  \includegraphics[width=0.9\linewidth]{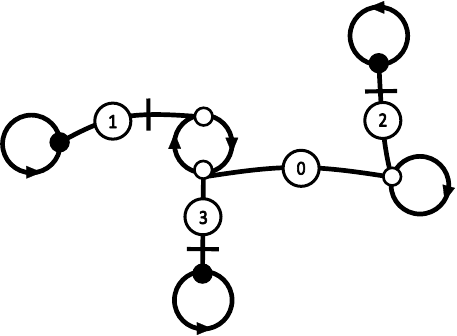}
  \caption{}
  \end{subfigure}
  \begin{subfigure}[b]{0.3\linewidth}
  \centering
  \includegraphics[width=0.9\linewidth]{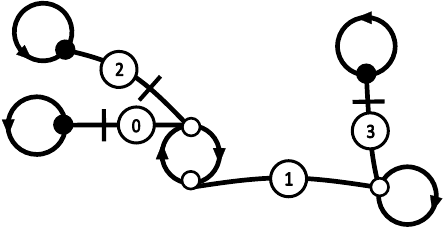}
  \caption{}
 \end{subfigure}
 \begin{subfigure}[b]{0.3\linewidth}
  \centering
  \includegraphics[width=0.9\linewidth]{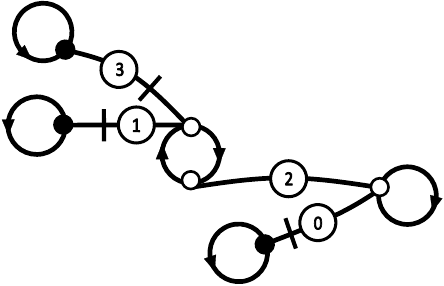}
  \caption{}
 \end{subfigure}
  \begin{subfigure}[b]{0.3\linewidth}
  \centering
  \includegraphics[width=0.9\linewidth]{mobileexample.pdf}
  \caption{}
 \end{subfigure}
  \caption{(a) The mobile $M$ of Figure \ref{fig:hurgalmarked}. (b)-(e) Successive applications of $\sigma$ to the mobile $M$.}\label{fig:shiftequivalence}
\end{figure}
\end{example}

\subsection{Hurwitz Dyck Paths}
In this subsection, we lay the groundwork towards understanding the relation between Hurwitz galaxies and Hurwitz mobiles. In order to achieve this, we introduce the notion of Hurwitz Dyck paths, which are a special class of Dyck paths that appear in the combinatorics of Hurwitz mobiles. These objects build on ideas of \cite{duchi2014bijections} and prove to be closely related to the classes of Dyck path found in \cite{dyckmodk}. In \cref{sec:bij} we provide a bijection between these objects and Hurwitz galaxies and mobiles, before using this bijection to provide a classification of pruned Hurwitz mobiles in \cref{sec:prunhurdyck}.

\begin{definition}\label{def:dyckpath}
    For some positive integer $n$, a Dyck path of length $2n$ is a lattice path consisting of $n$ up-steps of the form $u=(1,1)$ as well as $n$ down-steps of the form $d=(1,-1)$ which starts at $(0,0)$, ends at $(2n,0)$ and stays (weakly) above the $x$-axis.
\end{definition}
Additionally, an occurrence of $ud$ (resp. $du$) within a Dyck path is called a \textit{peak} (resp. \textit{valley}), an occurrence of $u^h d^h$ is called a \textit{pyramid of height} $h$, a down-step beginning at height $1$ is called a \textit{return step} and for any up-step $U$ its \textit{matching down-step} is the nearest down-step (to the right) such that the number of up-steps and down-steps between them are equal. We say that a Dyck path with only one return step is \textit{primitive}. For a vertex $v$, we denote its $x$-coordinate by $v_x$ and its $y$-coordinate by $v_y$.

Before defining single Hurwitz Dyck paths, we introduce two new notions.

\begin{definition}\label{def:kdyckpath}
     Let $b$ be a positive integer and $n$ some multiple of $b$. A $b$-Dyck path of length $2n$ is a Dyck path of length $2n$ in which the length of all maximal increases (a maximal path of upsteps) is a multiple of $b$.\\
     A vertex $v$ is called \textit{distinguished} if the number of up-steps in the Dyck path prior to $v$ is a multiple of $b$ and $v$ is succeeded by an up-step. Let $m$ be the number of distinguished vertices, then we label the distinguished vertices by $1,\dots,m$.
 \end{definition}

\begin{definition}\label{def:tdyckpath}
    Let $t$ be a positive integer. A $t$-marked Dyck path is a Dyck path in which $t$ pairs of vertices of down-steps have been marked, such that for any pair $(v,w)$ the $y$-coordinates of $v$ and $w$ co-incide. Furthermore, the $t$ pairs are labelled $1,\dots,t$.\\
    Given a pair $(v,w)$, we define its interval $[v,w]$ to be the lattice paths from $v$ to $w$ in the Dyck path. We define its \textit{essential interval} $[v,w]^e$ as 
    \begin{equation}
            [v,w]^e=[v,w]\backslash \bigcup_{(v',w')}[v',w']
    \end{equation}
    where the union runs over all marked pairs $(v',w')$ in $[v,w]$.
    For a marked pair $(v,w)$ in a $b$-Dyck path, we define its degree as 
    \begin{equation}
        \mathrm{deg}((v,w))=\frac{|\{\textrm{up-steps in }[v,w]^e\}|}{b}=\frac{|\{\textrm{down-steps in }[v,w]^e\}|}{b}.
    \end{equation}
\end{definition}

We are now ready to define single Hurwitz Dyck paths.

\begin{definition}\label{def:singlehurwitzdyck} 
    Let $d$ be a positive integer and $\mu$ be an ordered partition of $d$. Furthermore let $b=\ell(\mu)+d-2$.\\
    A \textit{single Hurwitz Dyck} path of type $\mu$ is then a primitive $\ell(\mu)$-marked $b$-Dyck path of length $2bd$, such that:
    \begin{enumerate}
        \item the pair $((0,0),(2bd,0))$ is marked,
        \item the set of $y$-coordinates of marked pairs modulo $b+1$ is of size $\ell(\mu)$,
        \item the set of $y$-coordinate of distinguished vertices modulo $b+1$ is of size $d$,
        \item the set of non-zero $y$-coordinates of marked pairs modulo $b+1$ and  the set of non-zero $y$-coordinate of distinguished vertices modulo $b+1$ are disjoint,
        \item for any two marked pairs of vertices $(v_1,w_1)$ and $(v_2,w_2)$, assuming w.l.o.g. that $(v_1)_x\le (w_1)_x$ and $(v_2)_x\le(w_2)_x$, we have that if $(v_1)_x\le (v_2)_x$, then $(w_2)_x\le(w_1)_x$. In other words, pairs of vertices admit a non-crossing condition,
        \item the marked pair labelled $i$ has degree $\mu_i$.
    \end{enumerate}
    We denote the set of Hurwitz Dyck paths by $D(\mu)$.
\end{definition}

\begin{example}
In \cref{fig:contourfunctionexample}, we illustrate a single Hurwitz Dyck path of type $(2,1)$. The distinguished vertices are marked red. The green vertices represent the marked pairs in this example. Note that $b=3$.

    \begin{figure}
  \centering
  \def\svgwidth{.6\textwidth}
  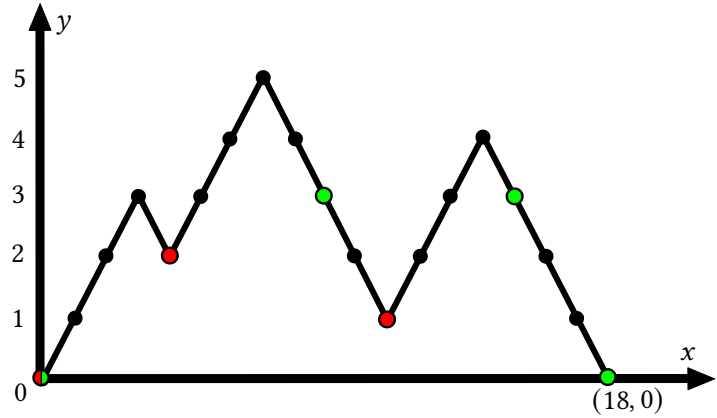
  \caption{A single Hurwitz Dyck path of type $\mu=(2,1)$.}\label{fig:contourfunctionexample}
\end{figure}
\end{example}
\subsection{Basics on Cacti}
In \cref{sec:bij}, we outline the bijection between marked Hurwitz galaxies of genus $0$ and shift equivalence classes of mobiles given in \cite{duchi2014bijections}. While the authors of \textit{loc.cit} construct an explicit map $\Phi\colon\mathcal{G}_0(\mu,\nu)\to\mathcal{M}(\mu,\nu)$, the proof of its bijectivity moves through combinatorial objects called \textit{cacti}. We now review some of the basic notions.

\begin{definition} \label{def:cactus}
\emph{\cite{duchi2014bijections}} We denote by $C_g(\mu,\nu)$ the set of graphs on surfaces of genus $g$ with one boundary such that the following hold:
\begin{enumerate}
    \item (Face colour condition) There are $l(\mu)$ white faces and $l(\nu)$ black faces, labelled $1,\dots,\ell(\mu)$ and $1,\dots,\ell(\nu)$. The white face labelled $i$ has degree $\mu_i$. Similarly, the degree of the black face labelled $j$ is $\nu_j$.
    \item There are three types of edges. These are:
    \begin{itemize}
        \item Internal edges, that are incident to a black face and a white face.
        \item White boundary edges, that are oriented and have a white face on their right hand side.
        \item Black boundary edges, that are oriented and have a black face on their left hand side.
    \end{itemize}
    We say that a vertex is active if it has at least one incoming white boundary edge.
    \item (Vertex colour condition) All vertices are incident to the boundary and have a colour in $\{ 0, \dots, b \}$. Each boundary edge $u \rightarrow v$ that joins a vertex $u$ with colour $c(u)$ to a vertex $v$ with colour $c(v)$ satisfies $c(v)=c(u)+1\,\mathrm{mod}\,b+1$.
    \item (Hurwitz condition) There are $d-1$ active vertices of each colour. 
\end{enumerate}
We call an element $C$ of $C_g(\mu,\nu)$ a cactus of type $(g,\mu,\nu)$.
\end{definition}

\begin{example}
\label{ex:cactusex}
    Let $g=0$, $\mu=(2,1)$ and $\nu=(1,1,1)$. A cactus of type $(0,(2,1),(1,1,1))$ is depicted in \cref{fig:cactusexample}.

\begin{figure}[H]
  \centering
  \includegraphics[width=0.5\linewidth]{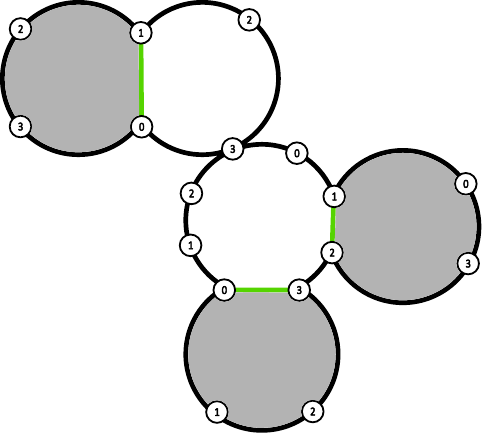}
  \caption{An example of a cactus $C$ of type $(0,(2,1),(1,1,1))$. Internal edges are shown in green. The vertex colours are included within the small circles along the boundary of $C$.}\label{fig:cactusexample}
\end{figure}
\end{example}

One can see that for cacti in $C_0(\mu,\nu)$, like $C$ in Figure \ref{fig:cactusexample}, the graph consists of polygons that are arranged to form a tree-like structure (a kind of cactus -- hence the terminology). Due to this tree-like structure, as well as the black and white galaxy-like faces, we see that cacti share some of the important properties of both Hurwitz galaxies and Hurwitz mobiles. As a first step, we associate a cactus to a marked Hurwitz galaxy. We begin with the following definitions.

\begin{definition}
Let $G\in\mathcal{G}_0(\mu,\nu)$ and $v$ be a vertex with two incoming geodesic edges. We define a splitting of $v$ in $G$ as the graph $\tilde{G}$ obtained by replacing $v$ by two new vertices, each carrying one incoming geodesic edge and the outgoing edge following it in clockwise direction around $v$. 
\end{definition}

\begin{definition}
Let $G\in\mathcal{G}_0(\mu,\nu)$ be a marked Hurwitz galaxy. Then, we define $\Theta(G)$ to be the marked graph obtained from $G$ by splitting all vertices with two incoming geodesic edges and removing non-geodesic edges.
\end{definition}

\begin{example}
    Let $G$ be the marked Hurwitz galaxy in \cref{fig:hurgalmarked}. The graph $\Theta(G)$ is illustrated in \cref{fig:treegalaxy}.

\begin{figure}
  \centering
  \includegraphics[width=0.5\linewidth]{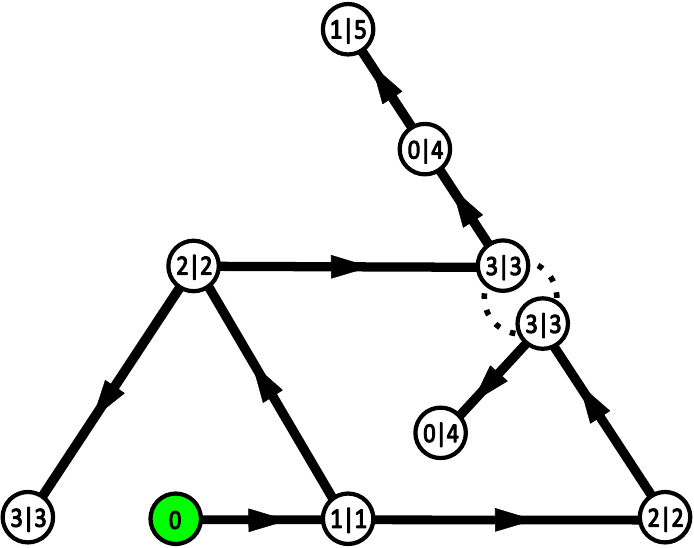}
  \caption{The construction of the tree $\Theta(G)$ corresponding to the marked galaxy $G$ of \cref{fig:hurgalmarked}. Notice the splitting of the 4-valent vertex of colour $3$, as well as the distance labelling which agrees with the distance labelling in  \cref{fig:hurgalmarked}.}\label{fig:treegalaxy}
\end{figure}
\end{example}

Note that by \cite[Proposition 3]{duchi2014bijections}, the graph $\Theta(G)$ is a tree  and for each vertex $v\in\Theta(G)$ the distance $\delta(v)$ (defined in the Hurwitz galaxy) is the distance of $v$ to the marked vertex in $\Theta(G)$.

\begin{construction}
\label{constr:galtocac}
    Let $G$ be a marked Hurwitz galaxy on some compact oriented surface $S$ of genus $g$. 
    \begin{enumerate}
        \item Consider the graph $\Theta(G)$ on $S$
        \item Since $\Theta(G)$ is a tree, we have that $S \setminus \Theta(G)$ has one open boundary and its closure is a surface $S^{\partial}$ of genus $g$ with one boundary.
        \item  We let $\Gamma(G)$ be the graph induced by $G$ on $S^{\partial}$.
    \end{enumerate} 
    By \cite[Lemma 4]{duchi2014bijections}, we have $\Gamma(G)\in C_g(\mu,\nu)$.
\end{construction}

We see that $\Gamma(G)$ contains faces and non-geodesic edges that it inherits from the faces and non-geodesic edges of $G$.\\
The geodesic edges of $G$ produce two boundary edges of $\Gamma(G)$, one white boundary edge and one black boundary edge. In total, we have the cases illustrated in \cref{fig:localcactus} (see also \cite[Section 4.1]{duchi2014bijections}).

\begin{figure}
  \centering
 \begin{subfigure}[b]{0.8\linewidth}
  \centering
  \includegraphics[width=0.75\linewidth]{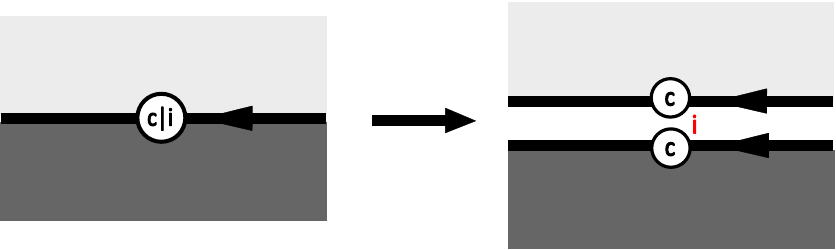}
  \caption{A vertex with one incoming geodesic edge produces two boundary edges in the construction of $\Gamma(G)$ (one white boundary edge and one black boundary edge).}
 \end{subfigure}
    \begin{subfigure}[b]{0.8\linewidth}
  \centering
  \includegraphics[width=0.75\linewidth]{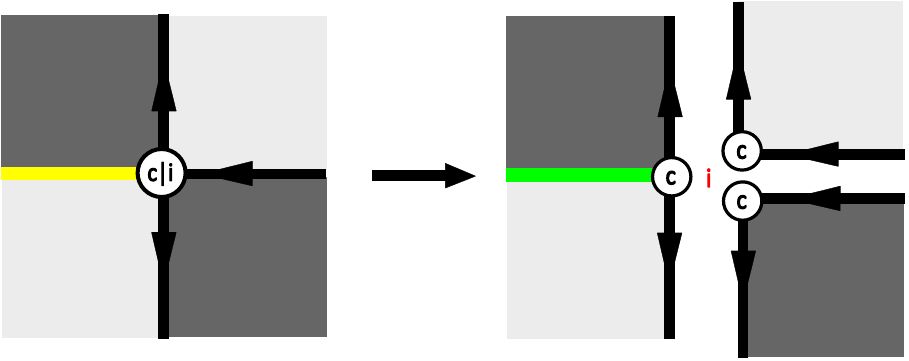}
  \caption{A 4-valent vertex with one incoming geodesic edge and one incoming non-geodesic edge produces the structure shown on the right. There are three vertices corresponding to the vertex of colour $c$, as well as an internal edge corresponding to the non-geodesic edge in $G$.}
  \end{subfigure}
  \begin{subfigure}[b]{0.8\linewidth}
  \centering
  \includegraphics[width=0.75\linewidth]{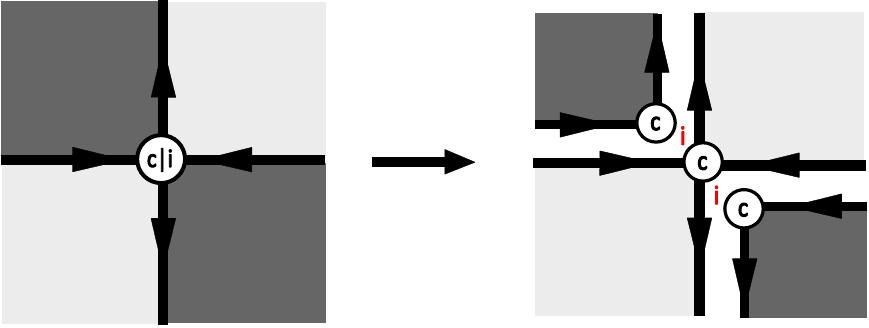}
  \caption{A 4-valent vertex with two incoming geodesic edges produces three vertices in $\Gamma(G)$; one of these is incident to two white polygons.}
  \end{subfigure}
  \caption{Local analysis of the configurations around vertices in a galaxy, as well as the structure that these configurations produce in the graph $\Gamma(G)$.}\label{fig:localcactus}
\end{figure}

 We call $(v,e,e')$ a \textit{boundary corner}, for a vertex $v$ with two adjacent edges $e$ and $e'$ traversing along the same part of the boundary .

\begin{example}
    The cactus considered in \cref{ex:cactusex} and illustrated in \cref{fig:cactusexample}, is obtained from the marked Hurwitz galaxy in \cref{fig:hurgalmarked} via \cref{constr:galtocac}. This is illustrated at the top of \cref{fig:galtocac}. At the bottom, we include the canonical corner labelling in red.
    \begin{figure}
        \centering
        \begin{subfigure}[b]{0.5\linewidth}
          \includegraphics[width=0.9\linewidth]{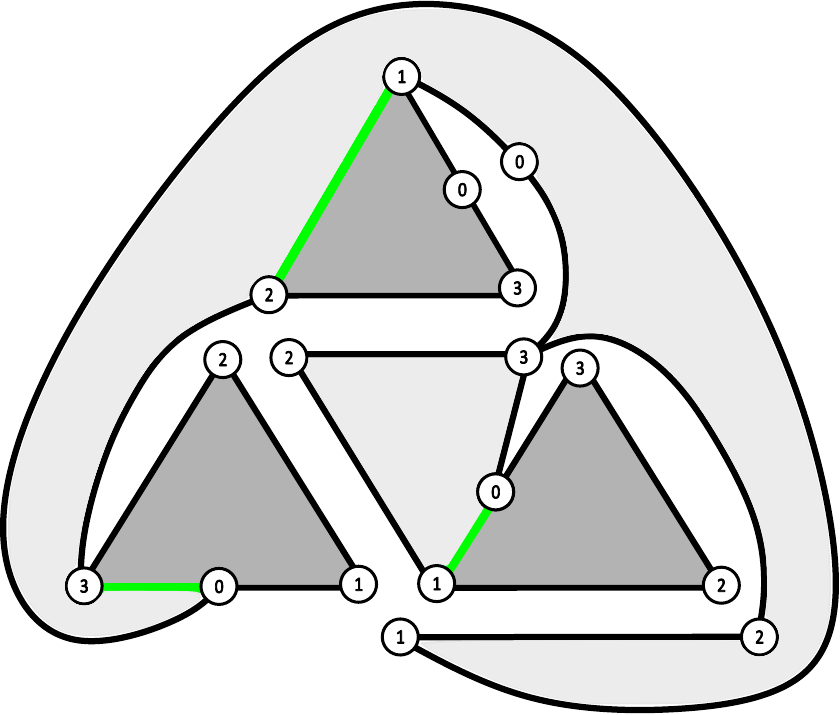}
          \caption{}
        \end{subfigure}
        \begin{subfigure}[b]{0.5\linewidth}
            \centering
            \includegraphics[width=0.9\linewidth]{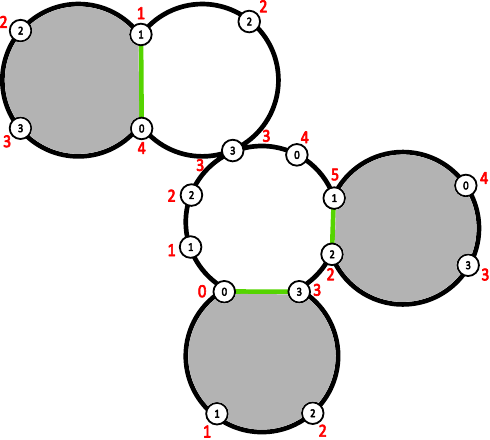}
            \caption{}
        \end{subfigure}
        \caption{(A): The cactus in \cref{constr:galtocac} obtained from the galaxy in \cref{fig:hurgalmarked}. (B): The cactus of \cref{fig:cactusexample} equipped with a canonical corner labelling.}
        \label{fig:galtocac}
    \end{figure}
\end{example}

\begin{definition} \label{def:canonicalcornerlabel}
Given a cactus $C$, a \textit{canonical corner labelling} is a mapping $\delta$ from the set of boundary corners of $C$ into the non-negative integers such that:
\begin{itemize}
    \item the minimum label is $0$, 
    \item for each boundary edge $e = u \rightarrow v$, $\delta(c^{\prime})=\delta(c)+1$, where $c$ is the boundary vertex incident to $e$ at $u$ and $c^{\prime}$ is the boundary vertex incident to $e$ at $v$.
\end{itemize}
\end{definition}

For any galaxy $G$, the corner labelling of $\Gamma(G)$ inherited from the distance labelling on $G$ is canonical by construction. Furthermore, any cactus $C\in C_g(\mu,\nu)$ has a unique canonical corner labelling by \cite[Lemma 5]{duchi2014bijections}.

\begin{definition} \label{def:coherent} \emph{\cite{duchi2014bijections}}
The canonical corner labelling $\delta$ of a cactus $C \in C_g(\mu,\nu)$ is said to be coherent if for each vertex $u \in C$, all boundary corners of $u$ have the same label. In this case the canonical corner labelling yields a vertex labelling called the \textit{coherent canonical labelling} of $C$. We denote by $C_g^c(\mu,\nu)$ the set of cacti of $C_g(\mu,\nu)$ whose canonical corner labelling is coherent.
\end{definition}

\begin{remark}
\label{rem:coherent}
    By \cite[Proposition 4]{duchi2014bijections}, we have $C_0(\mu,\nu)=C^c_0(\mu,\nu)$, i.e. all canonical corner labellings in genus zero are coherent.
\end{remark}

We end this section, with the following definition.
\begin{definition}
A cactus $C$ of $C_g^c(\mu,\nu)$ is said to be proper if the colour of its vertices with canonical label $0$ is $0$. We denote by $C_g^{0c}(\mu,\nu)$ the set of proper cacti. Elements of this set are said to have a proper canonical corner labelling.\\
\end{definition}

\section{Pruned Hurwitz numbers}
\label{sec:prunhur}
In \cite{do2018pruned}, a new variant of Hurwitz numbers was introduced, so-called \textit{pruned Hurwitz numbers}. Their motivation stem from the theory of Chekhov--Eynard--Orantin topological recursion. In this theory, one starts with a spectral curve as input datum and then recursively constructs a sequence of multi-differential forms. Astonishingly, for many enumerative invariants, one may find a spectral curve such that the coefficients of the multi-differentials obtained via topological recursion are exactly the enumerative invariants we started with.\\
For example, for the spectral curve
\begin{equation}
    x(z)=ze^{-z}\quad\mathrm{and}\quad y(z)=z,
\end{equation}
the multi-differentials $\omega_{g,n}$ that are the output of topological recursion satisfy

\begin{equation}
\omega_{g,n}=\sum_{\mu_1,\dots,\mu_n=1}^\infty\frac{{H}_g(\mu)}{b!}\prod_{i=1}^n\mu_ix_i^{\mu_i-1} \mathrm{d}x_1\cdots\mathrm{d}x_n,
\end{equation}
where $b$ is the number of simple branch points of the respective Hurwitz numbers.

This result is the content of the Bouchard-Mariño conjecture \cite{zbMATH05380331} that was proved in \cite{borot2011matrix,eynard2011laplace}.\\
The idea in \cite{do2018pruned} is to expand the differentials $\omega_{g,n}$ in the $z$-coordinates instead of the $x$-coordinates to obtain

\begin{equation}
\omega_{g,n}=\sum_{\mu_1,\dots,\mu_n=1}^\infty\frac{{PH}_g(\mu)}{b!}\prod_{i=1}^n\mu_iz_i^{\mu_i-1} \mathrm{d}z_1\cdots\mathrm{d}z_n.
\end{equation}

The astonishing observation of \textit{loc.cit.} is that the coefficients $PH_g(\mu)$ admit a natural interpretation as enumerative invariants that are now called \textit{pruned single Hurwitz numbers}.\\ 
We begin with their definition.

\begin{definition}
\label{def:pruned}
    Let $g$ be a non-negative integer, $d$ a positive integer and $\mu$ a partition of $d$. We define $PB_g(\mu)\subset B_g(\mu,(1^d))$ to be the subset of branching graphs of type $(g,\mu,1^d)$ with no $1$--valent vertices.\\
    We define the associated \textit{pruned single Hurwitz number} as
    \begin{equation}
        PH_g(\mu)=\sum_{[\Gamma]\in PB_g(\mu,1^d)}\frac{1}{|\mathrm{Aut}(\Gamma)|}.
    \end{equation}
\end{definition}

Next, we rephrase the definition of pruned Hurwitz numbers in terms of Hurwitz galaxies. To begin with, we need the following definition.

\begin{definition}
Let $P$ be some face of a Hurwitz galaxy $G$. We say that $P$ is a bubble if it contains exactly one 4-valent vertex $v$.
\end{definition}

\begin{example}
In \cref{fig:bubble1}, we illustrate a Hurwitz galaxy with a black bubble.
    \begin{figure}
        \centering
        \includegraphics[scale=0.6]{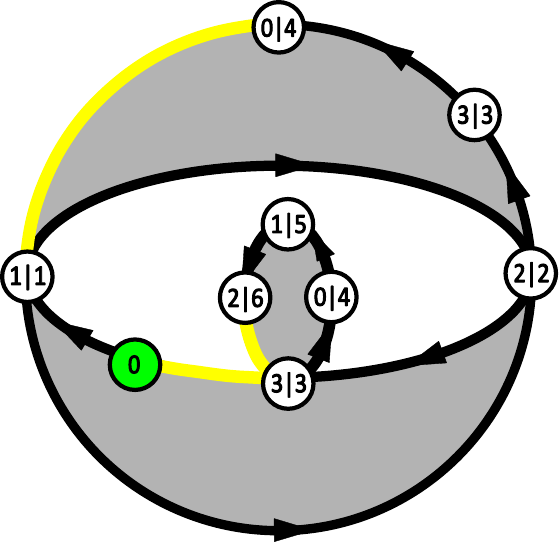}
        \caption{A Hurwitz galaxy of type $(0,(2,1),(1,1,1))$ with a bubble.}
        \label{fig:bubble1}
    \end{figure}
\end{example}

\begin{definition}
Let $g$ be a non-negative integer, $d$ a positive integer and $\mu,\nu$ partitions of $d$. Then, we define $PG_g(\mu)\subset G_g(\mu,1^d)$ to be the subset of Hurwitz galaxies of type $(g,\mu,\nu)$ without black faces that are bubbles.
\end{definition}

Recall from \cref{rem:brgalcorr} that there is a combinatorial correspondence between branching graphs and Hurwitz galaxies. It is easy to see that Hurwitz galaxies in $PG_g(\mu)$ exactly correspond to branching graphs in $PB_g(\mu)$ (\cite[Proposition 3.3]{hahn2020bi}) and thus that

\begin{equation}
        PH_g(\mu)=\sum_{[G]\in PG_g(\mu)}\frac{1}{|\mathrm{Aut}(G)|}.
\end{equation}

Next, we define a natural generalisation of $PH_g(\mu)$, namely pruned double Hurwitz numbers. For this it will be convenient to consider the pruning condition with respect to pre-images of $0$ instead of $\infty$.

\begin{definition}
    Let $g$ be a non-negative integer, $d$ a positive integer, $\mu,\nu$ partitions of $d$. Then, we define $PG_g(\mu,\nu)\subset G_g(\mu,\nu)$ to be subset of Hurwitz galaxies of type $(g\mu,\nu)$ without white faces that are bubbles. We call the elements of $PG_g(\mu,\nu)$ pruned Hurwitz galaxies.
    In particular, we define pruned double Hurwitz numbers as
    \begin{equation}
        PH_g(\mu,\nu)=\sum_{G\in PG_g(\mu,\nu)}\frac{1}{|\mathrm{Aut}(G)|}
    \end{equation}
    For some purposes, it is useful to ignore automorphisms and thus, we also define

\begin{equation}
        \widehat{PH}_g(\mu,\nu)=\sum_{[\Gamma]\in PB_g(\mu,\nu)}1.
    \end{equation} 
\end{definition}

We note that $PH_g(\mu,\nu)=\widehat{PH}_g(\mu,\nu)$ unless $\mu=\nu=(d)$. Moreover, we have $PH_g(\nu)=PH_g(1^d,\nu)$.

Pruned double Hurwitz numbers share many properties with their classical counterparts. Here we collect the ones relevant for our work. Firstly, pruned Hurwitz numbers may be expressed in terms of factorisations in the symmetric group. To formulate this correspondence, we need the following definition.

\begin{definition}
    For a permutation $\sigma \in S_d$, the \textit{support of} $\sigma$, denoted $\mathrm{supp}(\sigma)$, is the set of elements of $[d] = \{1, \ldots, d\}$ that are not fixed by $\sigma$.
\end{definition}

\begin{example}
    Consider the permutations $\sigma_1 = (246)\in S_7$ and $\sigma_2 = (23)\in S_7 $. Then, $$\mathrm{supp}(\sigma_1) = \{ 2,4,6\}, \quad 
    \mathrm{supp}(\sigma_2) = \{ 2, 3\}, \quad \text{ and } \mathrm{supp}(\sigma_1) \cap \mathrm{supp}(\sigma_2) = \{ 2\}.$$
\end{example}

Using this notation we may express pruned double Hurwitz numbers as the following weighted count.

\begin{theorem}[{\cite[Theorem 41]{zbMATH06791415}}]\label{def PDHN}
Let $g\geq 0$ and $d > 0$ be integers, and let $\mu, \nu$ be partitions of $d$. Then, we have
\begin{equation}
     PH_g(\mu, \nu)  = \frac{1}{d!} \cdot 
    \left|  \left\{ \begin{array}{l}
            (\sigma_1, \tau_1, \ldots, \tau_b, \sigma_2) \text{ such that:} \\
            \bullet \ b = 2g-2+\ell(\mu) + \ell(\nu), \\
            \bullet \ \sigma_1, \sigma_2, \tau_i \in S_d, \\
            \bullet \ \mathcal{C} (\sigma_1) = \mu,\ \mathcal{C}(\sigma_2) = \nu, \\ 
            \bullet \ \tau_i \text{ are transpositions }, \\
            \bullet \ \sigma_1 \tau_1 \cdots \tau_b = \sigma_2^{-1} \\
            \bullet \ \text{the subgroup generated by } (\sigma_1, \tau_1, \ldots, \tau_b, \sigma_2)\\
            \ \ \text{ acts transitively on } [d], \\
            \bullet \ \textrm{the cycles of }\sigma_1\textrm{ are labelled by }1,\dots,\ell(\mu),\\\textrm{such that the cycle labelled }i\textrm{ has length }\mu_i,\\
            \bullet \ \textrm{the cycles of }\sigma_2\textrm{ are labelled by }1,\dots,\ell(\nu),\\\textrm{such that the cycle labelled }i\textrm{ has length }\nu_i,\\
            \bullet \ \sum_{j=1}^b|\mathrm{supp}(\mu_i) \cap \mathrm{supp}(\tau_j)| \geq 2, \textrm{ for all } i = 1, \ldots , \ell(\mu).
          
    \end{array}\right\} \right|
\end{equation}    
\end{theorem}

\begin{remark}
    We note that the last condition differs from the one in \cite{zbMATH06791415}. Indeed, the given interpretation of pruned double Hurwitz numbers in \textit{loc. cit.} is not correct as the condition there misses loops in branching graphs. However, the two expressions only differ for $PH_0((\mu_1),(\nu_1,\nu_2))$. 
\end{remark}

Moreover, pruned double Hurwitz numbers admit polynomial behaviours reminiscent of that of classical Hurwitz numbers. More precisely, let $m,n>0$ integers and recall the hyperplane
\begin{equation}
    \mathcal{H}_{m,n}=\left\{(\mu,\nu)\in\mathbb{N}^m\times\mathbb{N}^n\mid\sum\mu_i=\sum\nu_j\right\}
\end{equation}
and the resonance arrangement inside this hyperplane
\begin{equation}
    \mathcal{R}_{m,n}=\left\{\sum_{i\in I}\mu_i-\sum_{j\in J}\nu_j=0\mid I\subset[m],J\subset[n]\right\}.
\end{equation}

We may now consider the map
\begin{align}
    PH_g\colon\mathcal{H}_{m,n}&\to\mathbb{Q}\\
    (\mu,\nu)&\mapsto PH_g(\mu,\nu).
\end{align}

\begin{theorem}[{\cite[Theorem 31]{zbMATH06791415}}]
    The map $PH_g$ is piecewise polynomial. More precisely, for each chamber $C$ of $\mathcal{R}_{m,n}$, there exists a polynomial $P_g^C$ in $m+n$ variables of degree $4g-3+m+n$, such that for all $(\mu,\nu)\in C$ we have $PH_g(\mu,\nu)=P_g^C(\mu,\nu)$.
\end{theorem}

Wall-crossing formulae reminiscent of the classical situation remain an open problem.

We end this section with the computation of the polynomial expression of certain pruned double Hurwitz numbers. The first two cases were already obtained in \cite[Example 36]{zbMATH06791415}.

 \begin{example}\label{base case pruned ex}
     \begin{itemize}
            \item 
            For $g=0$, $\ell(\mu) = 1$, and $\ell(\nu) = 2$, we have $$ PH_0((a), (b, c)) = 1.$$

            \item 
            For $g=0$, $\ell(\mu) = 2 = \ell(\nu)$, we have $$PH_0 ((a, b), (c,d) ) = 2 \cdot \min \{ a, b, c, d\}.$$

            \item Generalising Example 36 in \cite{zbMATH06791415} we consider the case $\ell(\mu)= n \geq 3$, $\ell(\nu) = 2$. That is, $\mu= (\mu_1, \ldots, \mu_n)$ and $\nu = (\nu_1, \nu_2)$. Then, 
            $$PH_0(\mu, \nu)  = \frac{1}{n} \cdot 
    \left|  \left\{ \begin{array}{l}
            (\alpha_1, \ldots, \alpha_n, i_1, \ldots, i_n, j_1, \ldots, j_n) \text{ such that:} \\
            \bullet \ (i_1, \ldots, i_n) , (j_1, \ldots, j_n) \text{ are ordered} \\
            \bullet \ \text{tuples with entries in } [n], \\
            \bullet \ \sum_{i=1}^n \alpha_i = \nu_1, \\
            \bullet \ 0 \leq \alpha_i \leq \mu_i, \\
            \bullet \ \text{If } j_k < j_{k-1}: \ \alpha_i \geq 1, \\ 
            \bullet \ \text{If } j_k > j_{k-1}: \ \alpha_i \leq \mu_{i-1}.
          
    \end{array}\right\} \right|.$$
    We see immediately that $PH_0(\mu,\nu)$ is polynomial in the resonance arrangement.
   

            \item We have $PH_0((a), (a)) = 0 $ and $PH_0((a, b), (c)) = 0 $.
            Furthermore, for all $PH_g(\mu, \nu)$ with $(g, \ell(\nu) ) = (0, 1)$ we find that
            $$ PH_g((\mu_1, \ldots, \mu_m), (\nu_1) ) = 0 .$$

     \end{itemize}
 \end{example}

\section{Pruned Hurwitz mobiles}
\label{sec:bij}
In this section, we express pruned single Hurwitz numbers in terms of pruned Hurwitz mobiles. To begin with, in \cref{sec:hurdyckbij} we recall the bijection  between marked Hurwitz galaxies of genus $0$ and shift equivalence classes of mobiles given in \cite{duchi2014bijections} and rephrase it in terms of of single Hurwitz Dyck paths. We then derive an expression of pruned single Hurwitz numbers in terms of a distinguished subset of single Hurwitz Dyck paths in \cref{sec:prunhurdyck}. Finally, we employ this expression to prove a correspondence between pruned Hurwitz galaxies and a distinguished subset of Hurwitz mobiles in \cref{sec:prunhurmob}, giving an interpretation of pruned single Hurwitz numbers in terms of what we call \textit{pruned Hurwitz mobiles}.

\subsection{Single Hurwitz numbers and Dyck paths}
\label{sec:hurdyckbij}
We construct a bijective map
\begin{equation}
    \Phi\colon \mathcal{G}_0(\mu,\nu)\to\mathcal{M}(\mu,\nu).
\end{equation}
between marked Hurwitz galaxies and shift equivalence classes of mobiles. Here, we focus on the case of $\nu=(1^d)$ and rephrase the bijection in terms of single Hurwitz Dyck paths, introduced in \cref{def:singlehurwitzdyck}.

\begin{construction} 
\label{constr:galtomob}
Let $G\in\mathcal{G}_0(\mu,\nu)$. The steps of the construction of $\Phi(G)$, as described in \cite[Section 2]{duchi2014bijections} are as follows:
\begin{enumerate}
    \item Polygons: Place in each face of degree $i$ of the galaxy $G$ an oriented $i$-gon -- oriented clockwise in the white faces and counter-clockwise in the black faces. These are the nodes and arcs of $\Phi(G)$, i.e the white and black polygons of the mobile.
    \item Construction lines: Join with dashed lines the $i$ nodes in each face $F$ of degree $i$ to the centres of the $i$ edges given by $b\rightarrow 0$ on the boundary of the face. This divides the face of degree $i$ into $i+1$ sub-regions: The interior of the white or black polygon, as well as $i$ sub-regions that each have on their boundary a path $0 \rightarrow 1 \rightarrow \dots \rightarrow b$ along the boundary of $F$, as well as dotted lines and an arc joining two adjacent nodes of the $i-$gon.
    \item Positive weight edges: For each non-geodesic edge $e = u \rightarrow v$ with weight $w(e)$, let $F_{\circ}$ and $F_{\bullet}$ be the white and black faces incident to $e$ respectively, and let $x$ (resp. $y$) denote the origin of the unique arc of the polygon incident to the same sub-region of $F_{\circ}$ (resp. $F_{\bullet}$) as the vertex $v$. Create an edge with label $c(v)$ and weight $w(e)$ between the nodes $x$ and $y$ (constructed through the edge $e$).
    \item Zero weight edges: For each vertex $v$ of $G$ with colour $c(v)$ that has two incoming geodesic edges, let $F_{\circ}$ and $F_{\circ}'$ denote the two incident white faces, and let $y$ (resp. $y^{\prime}$) denote the origin of the unique arc of the respective polygon incident to the same sub-region of $F_{\circ}$ (resp. $F_{\circ}'$) as $v$. Create an edge with label $c(v)$ and weight zero between $y$ and $y^{\prime}$ (constructed to pass through the vertex $v$).
    \item Remove the dashed lines.
\end{enumerate}
We call the resulting graph $\Phi(G)$.
\end{construction}

\begin{example}
    We illustrate \cref{constr:galtomob} in \cref{fig:galaxytomobile} for the marked galaxy of \cref{fig:hurgalmarked}. We obtain as a result the Hurwitz mobile in \cref{fig:hurmobile}.

    \begin{figure}
  \centering
  \includegraphics[width=0.6\linewidth]{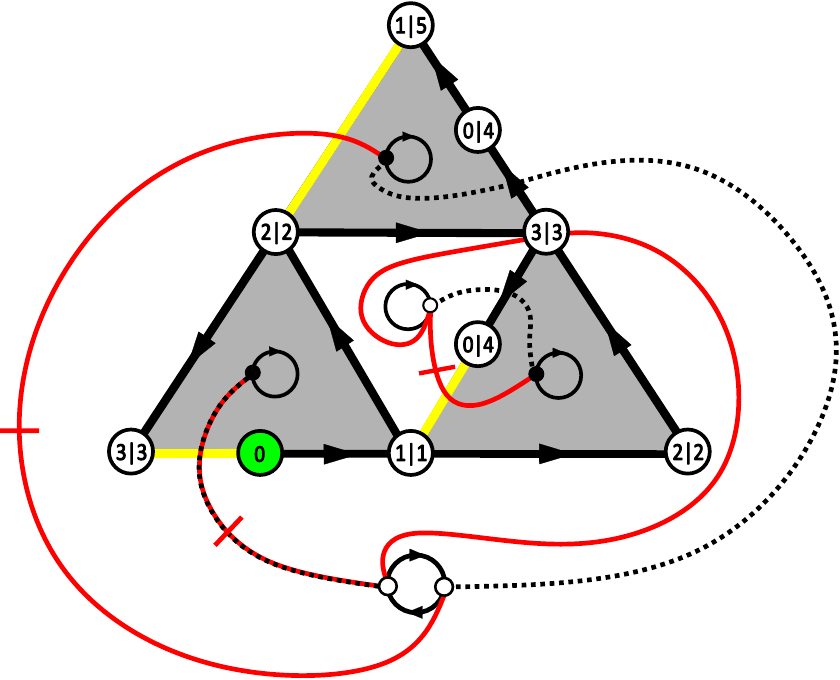}
  \caption{The construction of an edge-labelled Hurwitz mobile from the galaxy $G$ of type $(0,(2,1),(1,1,1))$ in \cref{fig:hurgalmarked}. The mobile $\Phi(G)$ that obtained from this construction (edges shown in red) is the mobile of type $(0,(2,1),(1,1,1))$ of \cref{fig:hurmobile}.}\label{fig:galaxytomobile}
\end{figure}
\end{example}

\begin{theorem}[{\cite[Theorem 1]{duchi2014bijections}}]
\label{thm-schaeffer}
    Let $\mu,\nu$ partitons of $d$. Then
    \begin{equation}
        \Phi\colon\mathcal{G}_0(\mu,\nu)\to\mathcal{M}(\mu,\nu)
    \end{equation}
    is bijective.
\end{theorem}

Since the proof of \cref{thm-schaeffer} is constructive, it will play an important role in our proof of \cref{thm:classhurmob}. Thus, we outline the key steps and refer to \cite[Section 4]{duchi2014bijections} for a more detailed account . Rather than proving explicitly that $\Phi$ is bijective, the authors of \cite{duchi2014bijections} split $\Phi$ into two steps, then prove that each is bijective. We translate parts of the proof into the language of Dyck paths.

As a first step, we consider the map $\Gamma$ associating a cactus to a marked Hurwitz galaxy. It was proved in \cite[Corollary 8]{duchi2014bijections} that $\Gamma$ is a bijection. As a first step we focus on the case $(0,\mu,(1^d))$ and rephrase the bijection in terms of single Hurwitz Dyck paths that we introduced in \cref{def:singlehurwitzdyck}.

\begin{construction} \label{constr:dyckfromcactus}
    Given some cactus $C \in C_0(\mu,1^d)$, we construct a lattice path as follows:
    \begin{enumerate}
        \item We begin at $(0,0)$ and construct the lattice path and move in a counterclockwise direction from the marked vertex along the boundary of $C$.
        \item For each black boundary edge, we insert an up-step and for each white boundary edge, we insert a down-step.
        \item For any vertex in $C$ incident to two white faces, we indicate the two instances it appears in the lattice path as a marked pair as in \cref{def:tdyckpath}.
        \item When crossing a vertex from a white to a black boundary edge, distinguish the corresponding lattice point.
        \item As $C$ only has one boundary forming a cyclic path, we end again at the marked vertex, which ends the construction. 
    \end{enumerate}
    We denote the constructed lattice path by $D(C)$.
\end{construction}

\begin{remark}
\label{rem:cornerlabel}
    We note that by construction, for a given vertex in a cactus $C$, the $y$-coordinate of corresponding vertex in $D(C)$ agrees with its canonical corner label.
\end{remark}

The constructed lattice paths lie in $D(\mu)$. More precisely, we have the following result.

 \begin{proposition}
      Given some cactus $C \in C_0(\mu,1^d)$, the lattice path $D(C)$ is a single Hurwitz Dyck path of type $\mu$.\\
 \end{proposition}

\begin{proof}
We first show that for a given cactus $C$, the lattice paths $D(C)$ lies in $D(\mu)$. To show that $D(C)$ is a Dyck path, we need to prove that it does not have vertices with negative $y$-coordinate. This follows from \cref{rem:cornerlabel}. The fact that it is a $b$-Dyck path follows from the fact that each black face in $C$ possesses exactly $b$ black boundary edges and precisely one internal edge, since the sum of weights of edges of a face sum up to its degree in a galaxy. The Dyck path is $\ell(\mu)$ marked, since we have $\ell(\mu)+d-1$ colours in a galaxy (and thus in a cactus) of type $(0,\mu,(1^d))$. These give rise to $\ell(\mu)+d-2$ many $4-$valent vertices and the unique marked vertex. Each black face has a unique internal edge in the cactus corresponding to a non-geodesic edge in the galaxy. As we have $d$ black faces, we have $\ell(\mu)-1$ many $4$-valent vertices in a galaxy without incoming non-geodesic edges. These correspond to the vertices in $C$ incident to two white faces. Thus, we obtain $\ell(\mu)-1$ pairs of markings. In addition, we have marked the begining and end of the Dyck path. The fact that each element of the pair has the same $y$-coordiante follows from the fact that the canonical corner labelling is coherent. In particular, this implies conditions (1) and (2) of \cref{def:singlehurwitzdyck}. By the same consideration, we obtain conditions (3) and (4). Conditions (5) and (6) follow by construction and the fact that the boundary of a cactus forms a single cyclic path.
\end{proof}

Next, we consider the opposite direction.

\begin{construction}
\label{constr:dyckcacti}
    Let $D \in D(\mu)$ be a single Hurwitz Dyck path. We construct a cactus from $D$.
    \begin{enumerate}
        \item We fix $d$ black polygons of degree $1$ and $\ell(\mu)$ white polygons of degrees $\mu_i$, labelled $1,\dots,d$ and $1,\dots,\ell(\mu)$ respectively. The white polygon labelled $i$ consists of $(b+1)\mu_i$ vertices coloured clockwise $0,\dots,b,0,\dots$. Similarly, the black polygon labelled $j$ consists of $b+1$ vertices coloured counterclockwise $0\dots,b$.
        \item Let $i$ be the label of $((0,0),(2bd,0))$ and consider the white polygon labelled $i$.
        \item We choose a vertex $v$ of colour $0$ in $F$ to be the marked vertex.
        \item We consider the first sequence of $b$ up-steps, connecting $(0,0)$ to $(b,b)$, and consider the corresponding distinguished vertex labelled $j$. We glue the edge connecting a vertex of colour $b$ to a vertex of colour $0$ of the black polygon labelled $j$ to the edge $w\to v$ in $F$ connecting a vertex of colour $b$ to $0$ as well. Furthermore, we consider the glued edge to be an internal edge.
        \item Let the next distinguished/marked vertex of $D$ occur after $k$ down-steps ($0 \leq k \leq b$). We have two possible cases:
        \begin{enumerate}
            \item this vertex is a marked vertex labelled $\tilde{i}$. In this case, move in a counterclockwise direction along $k$ edges of the polygon $F$, starting at the vertex $w$ and ending at some vertex $\tilde{w}$. Glue at this vertex a white face of degree $\mu_{\tilde{i}}$ (so that the colour of the glued vertices coincide). 
            \item this vertex is a distinguished vertex labelled $\tilde{j}$. As in step 4, we attach the black polygon labelled $j$ to the white face. The gluing is performed to the edge of the white face, that is $k$ steps in clockwise direction from the previous internal edge with matching colours.
        \end{enumerate}
        \item We now iteratively repeat step 5 by traversing along white faces for down-steps, gluing in white faces (and traversing along them) for marked vertices and gluing in black faces for distinguished vertices. Note, that after meeting the second vertex of a marked pair, we then traverse along the already constructed corresponding white face.
    \end{enumerate}
\end{construction}

\begin{example}
In \cref{fig:dycktocactus}, we illustrate \cref{constr:dyckcacti} for the Dyck path in \cref{constr:dyckcacti}. We see that the first marked pair $((0,0),(18,0))$ has degree $2$. Thus, we start with a white polygon of degree $2$ with a fixed marked vertex.\\
Traversing along the Dyck path, we first meet a distinguished vertex at height $0$ which corresponds to gluing in the black polygon in the second step. We encounter the next distinguished vertex at height $2$. Thus, this corresponds to the gluing of the black polygon in the third step. Traversing further, we encounter the first vertex of a marked pair at height $3$, which gives rise to the gluing of a white face in the fourth step. The third distinguished vertex at height $1$ finally yields the gluing in the last step. Traversing further along the Dyck path, confirms that this is the entire cactus and we end up again at the marked vertex.


 \begin{figure}
  \centering
  \includegraphics[width=0.95\linewidth]{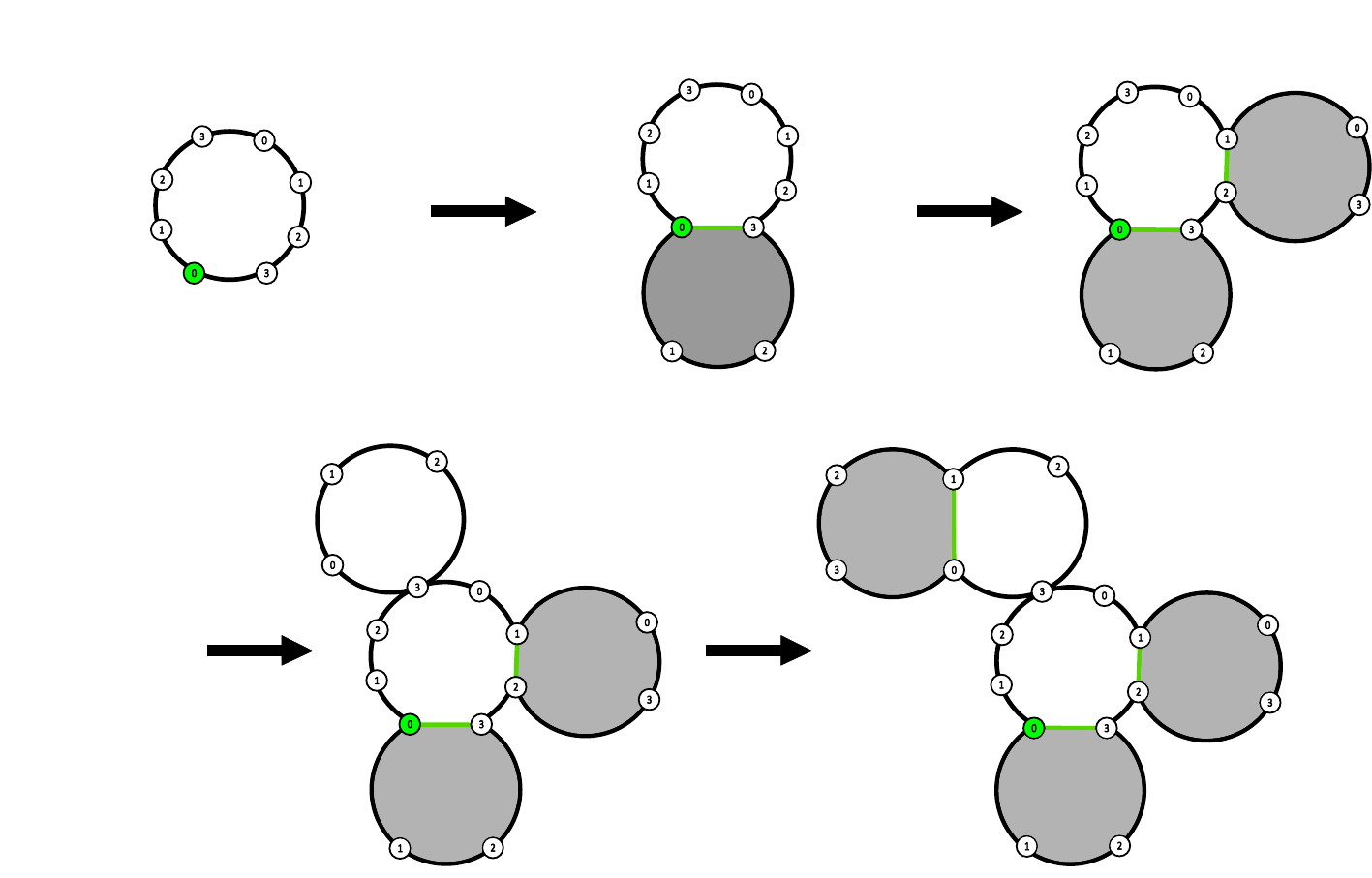}
    \caption{We build a cactus iteratively for the Dyck path in \cref{fig:contourfunctionexample}}\label{fig:dycktocactus}
\end{figure}
\end{example}

The following proposition follows by construction.

 \begin{proposition}
     Given a single Hurwitz Dyck path $D\in D(\mu)$, the cactus $\mathcal{C}(D)$ lies in $C_0(\mu,(1^d))$. In particular, the map
      \begin{align}
          D\colon C_0(\mu,1^d)&\to D(\mu)\\
          C&\mapsto D(C)
      \end{align}
      is a bijection.
 \end{proposition}

\begin{remark}
 We note that traversing along the Dyck path, corresponds to traversing along the boundary of the cactus. The vertices of the Dyck path that are identified in the cactus are exactly those that are connected by a horizontal line that lies weakly below the Dyck path.
\end{remark}

Next, we describe how to construct Hurwitz mobiles from cacti. Recall that given a cactus $C\in C_0(\mu,\nu)$, its internal edges correspond to non-geodesic edges in its corresponding Hurwitz galaxy and the vertices incident to two white faces correspond to the vertices in the galaxy that split. Thus, we can apply \cref{constr:galtomob} again to obtain a Hurwitz mobile $\Pi(C)$. The following theorem was proved in \cite[Proposition 7]{duchi2014bijections}.

\begin{theorem}
    The map
    \begin{align}
        \Pi\colon C_0(\mu,(1^d))&\to\mathcal{M}(\mu,(1^d))\\
        C&\mapsto [\Pi(C)]
    \end{align}
    is a bijection.
\end{theorem}

We remark that the shift equivalence classes of the Hurwitz mobiles correspond to shifting the colour of each vertex in a cactus by $1$ mod $b+1$. We will see immediately from the construction that \cref{constr:dyckcacti} can be easily adapted to construct Hurwitz mobiles from single Hurwitz Dyck paths.

Finally, we discuss how Dyck paths encode the procedure that describes how to re-glue a cactus to a Hurwitz galaxy.

\begin{definition}
   Let $D\in D(\mu)$ a single Hurwitz Dyck path. Let $v$ and $w$ be lattice points on $D$ with the same height. Let $H$ be the horizontal line connecting $v$ and $w$. We call $H$ a \textit{gluing line}, if
   \begin{enumerate}
       \item $v$ and $w$ are not valleys in $D$,
       \item the only other vertices of $D$ on $H$ are valleys.
   \end{enumerate}
\end{definition}

We note that gluing lines are exactly those maximal horizontal lines that live under the Dyck path. The following lemma was essentially proved in the proof of \cite[Proposition 5]{duchi2014bijections}.

\begin{lemma}
\label{lem:gluehor}
Let $D\in D(\mu)$ a single Hurwitz Dyck path and $C$ its corresponding cactus via \cref{constr:dyckcacti}. Moreover, let $G=\Gamma^{-1}(C)$ the marked Hurwitz galaxy corresponding to $C$.\\
Then, $G$ is obtained from $C$ by gluing those tuples of vertices that lie on a shared gluing line.
\end{lemma}

\subsection{Pruned Hurwitz Dyck Paths and pruned Hurwitz mobiles}
\label{sec:prunhurdyck}
In this subsection, we express pruned single Hurwitz numbers in terms of \textit{pruned Hurwitz mobiles}. We begin by deriving a classification of \textit{pruned single Hurwitz Dyck paths}, where we call a single Hurwitz Dyck path $D$ pruned if it gives rise to a pruned Hurwitz galaxy.

\begin{proposition}[Classification of pruned single Hurwitz Dyck paths]
    A single Hurwitz Dyck path $D$ of type $\mu$ is pruned if and only if the following conditions are satisfied:
    \begin{itemize}
         \item (descent condition) any sequence of down-steps of the form $d^b$ that begins at a peak of $D$ must contain a vertex belonging to a marked pair.
        \item (lowest distinguished vertex condition) either the distinguished vertex $v$ of lowest non-zero height in $D$ is not the second distinguished vertex of $D$  (i.e. the only distinguished vertex to the left of $v$ in $D$ is the origin) or if it is the second distinguished vertex of $D$, then there must be a vertex belonging to a marked pair to the left of $v$.
    \end{itemize}
\end{proposition}

\begin{proof}
If the descent condition and the lowest distinguished vertex condition are satisfied, then clearly any black face carries at least two four valent vertices in the corresponding galaxy.\\
For the other direction, let $G$ be a pruned Hurwitz galaxy and $D(G)$ be its corresponding single Hurwitz Dyck path. Recall that any black face of a galaxy is a bubble if and only if it contains exactly one $4$-valent vertex. Therefore, any black face in $G$ contains at least two $4$ valent vertices. As outlined in \cref{constr:dyckcacti} each sequence of $b$ up-steps starting at a distinguished vertex corresponds to a black face of the cactus $C$ and thus the galaxy $G$ that we started with. Moreover, the vertices of $D(G)$ corresponding to the 4-valent vertices of $G$ are precisely the distinguished vertices (excluding the origin) and the marked pairs (excluding the pair $((0,0),(2bd,0))$). The vertices of $D$ corresponding to the marked vertex of $G$ are $(0,0)$ and $(2bd,0)$. By \cref{lem:gluehor}, all vertices of $G$ are obtained by gluing lattice points on gluing lines of $D$. 
If the descent condition is violated then the preceeding black face is not glued to a marked pair or an extra distinguished vertex and therefore only contains one $4$-valent vertex corresponding to the single distinguished vertex. If the lowest distinguished vertex condition is violated, then the black face in $G$ which contains the marked vertex contains just one 4-valent vertex, since it is glued to no vertices belonging to marked pairs, and to a single distinguished vertex, namely the second one. In either case, we obtain a bubble, which is a contradiction, since $G$ was assumed to the pruned.
\end{proof}

Before, we can state our main theorem of this section, we have to introduce several notions.

\begin{definition}
\label{def:mobiledistance}
Let $P$ be a black polygon in a single Hurwitz mobile $M$, and let $y$ be the weight $1$ edge incident to $P$. Furthermore let $z$ be some other labelled edge of $M$. We define the white distance $d_{\circ}(y,z)$ (resp. black distance $d_{\bullet}(y,z)$) between $y$ and $z$ to be the number of white polygon arcs (resp. black polygon arcs) traversed in a counterclockwise path that starts at $y$, traverses along the black arc belonging to $P$, continues along $M$ and ends at $z$.

\end{definition}

\begin{definition}
\label{def:distance}
Let $P$ be a black polygon in a single Hurwitz mobile $M$, and let $y$ be the weight $1$ edge incident to $P$. We say that the edge labelled $y$ is \textit{interrupted} by some edge $z$ of $M$ if: 

\begin{itemize}
    \item $d_{\circ}(y,z)<d_{\bullet}(y,z)$, or,
    \item $\dw(y,z)=\db(y,z)$ with $y<z$.
\end{itemize}
\end{definition}

\begin{example}
    We consider the Hurwitz mobile depicted in \cref{fig:hurmobile}. Then, we have e.g. $\dw(0,2)=\dw(2,1)=\dw(1,0)=1$ and $\dw(0,1)=2$. Moreover, we have $\db(0,2)=\db(2,1)=\db(1,0)=1$ and $\db(0,1)=2$. Then, we see that the edge labelled $0$ is interrupted by the edges labelled $2$ and $1$ since $\dw(0,2)=\db(0,2)=1$ and $\dw(0,1)=\db(0,1)=2$ with $0<1,2$. On the other hand, the edges labelled $2$ and $1$ are not interrupted by the edges labelled $1$ and $0$ respectively since $2>1>0$.
\end{example}

\begin{definition}
    We call a single Hurwitz mobile $M$ pruned if it corresponds to a pruned Hurwitz galaxy. Moreover, we call $M$ in standard form if it is the element of its shift equivalence class that is the image of a marked Hurwitz galaxy under \cref{constr:galtomob}.
\end{definition}

\begin{theorem}[Classification of pruned single Hurwitz mobiles]
\label{thm:classhurmob}
    A single Hurwitz mobile $M$ in standard form of type $\mu$ is pruned if and only if each black polygon $P$ of $M$ satisfies either of the following two conditions:
    \begin{enumerate}
    \item The edge of weight $1$ that is incident to $P$ has label $y \neq 0$ and is interrupted by the next labelled edge $z$ in $M$ (i.e. the first edge reached in a counterclockwise path starting at $y$).
    \item The edge of weight $1$ that is incident to $P$ has label $0$ and:
    \begin{itemize}
        \item[(a)] working counter-clockwise from the edge $0$, the next labelled edge in the mobile is a non-weighted edge, or
        \item[(b)] the next labelled edge in the mobile (again working counter-clockwise from the edge $0$) is a weight $1$ edge of label $y$ and there exists some edge of label $z \neq 0$ in $M$ that does not interrupt $y$.
        
    \end{itemize}
\end{enumerate}
\end{theorem}

\begin{proof}
    Let $M$ be a Hurwitz mobile of type $(\mu,1^d)$, $G$ the corresponding Hurwitz galaxy and $D$ the respective Dyck path. We show that $M$ satisfies the first condition if and only if $D$ satisfies the descent condition. Moreover, we prove that $M$ satisfies the second condition if and only if $D$ satisfies the lowest distinguished vertex condition. Then, the theorem follows.\\
    Let $M$ satisfy the first condition. Let $P$ be a black polygon in $M$ and $\tilde{u}$ the corresponding sequence of $b$ up-steps in $D$. Let $\tilde{d}$ be the set of $b$ down-steps in $D$ that is glued to $\tilde{u}$ along horizontal lines. Then, the last lattice point of $\tilde{d}$ has label $y\,\mathrm{mod}\,(b+1)$. If $\tilde{d}$ is not a connected sequence then there is nothing to prove. If it is a connected sequence, then we see that it must contain a vertex belonging to a marked pair.
    This is due to the fact that the number of down-steps traversed before the next distinguished vertex or marked pair is precisely $(b+1)\cdot \dw(y,z)+y-z-\db(y,z)$.
    \begin{itemize}
        \item If $\dw(y,z)=0$, then we have $y>z$ since the two edges have to live on the same node. Then, we have that $(b+1)\cdot \dw(y,z)+y-z-\db(y,z)\le y-z\leq b$. We cannot $y-z=b$ since $4$-valent vertices must have distinct labels in galaxies.
        \item If $\dw(y,z)=1$ but $z>y$, we have $(b+1)\cdot \dw(y,z)+y-z-\db(y,z)= b+1+y-z-\db(y,z)$. Since the edge $y$ is incident to a black face, we have $\db(y,z)\ge1$ (in fact it is equal to $1$, since $z$ is the first edge interrupting $y$.). Therefore, we obtain with $z>y$, that $b+1+y-z-\db(y,z)<b$.
    \end{itemize}
     Note that $\dw(y,z)$ cannot be greater than $1$ as $\db(y,z)=1$ and $y$ is necessarily interrupted by $z$. Thus, in the two possible cases, the next distinguished vertex or vertex belonging to a marked pair is part of the $\tilde{d}$ down-steps. Finally, we note that any sequence of down-steps as considered in the descent condition arises this way. This proves one direction. The reverse direction is obtained by working the same steps backwards.\\
    Now, let $M$ satisfy the second condition. Moreover, let $P$ be the black polygon that is incident to the edge of label $0$. Let $v$ be the lowest distinguished vertex in $D$ corresponding to the edge of label $z$. If $v$ is not the second distinguished vertex in $D$, there is nothing to prove. If it is the second distinguished vertex, then we prove that $M$ has to satisfy case (a) which immediately proves the assertion that there is a marked pair to the left of $v$. Therefore, assume that $v$ is the second distinguished vertex in $D$ and that $M$ satisfies case (b) of the second condition.\\
    The edge of label $y$ in case (b) corresponds to the vertex $v$. Let $z$ be an edge that does not interrupt $y$ and let $w$ the corresponding vertex. We note that $w$ is on the right of $v$. Considering the lattice path on $D$ from $v$ to $w$, we define $D_{\textrm{up}}(w,v)$ the number of upsteps on this path and $D_{\textrm{down}}(w,v)$ the number of down-steps. Either $w$ is a distinguished vertex, then clearly, we have
    \begin{equation}
    \label{equ:inequ}
        D_{\textrm{up}}(w,v)-D_{\textrm{down}}(w,v)>0.
    \end{equation}
    If $w$ is a marked vertex, its $y$-coordinate can only be smaller than the $y$-coordinate of $v$ if it comes from a pair with one lattice point on the first sequence of up-steps of $D$ and the other on the last sequence of down-steps. This is however not possible, since $y$ is the first labelled edge after $0$. Thus, we again have \cref{equ:inequ}.
    Moreover, we observe that
    \begin{equation}
        D_{\textrm{down}}(w,v)=(b+1)\cdot\dw(y,z)+y-z-\db(y,z)\quad\textrm{and}\quad  D_{\textrm{up}}(w,v)=b\cdot \db(y,z).
    \end{equation}
    Therefore, in total, we obtain
    \begin{equation}
        0<D_{\textrm{up}}(w,v)-D_{\textrm{down}}(w,v)=(b+1)(\db(y,z)-\dw(y,z))+y-z.
    \end{equation}
    However, since $z,y\neq0$, we have that $|z-y|<b$. Therefore, we either have $\db(y,z)>\dw(y,z)$ or $\db(y,z)=\dw(y,z)$ and $y<z$. In both cases, we obtain that $z$ does interrupt $y$ a contradiction. Therefore, we are in case (a). This proves that the second condition in the theorem implies the lowest distinguished vertex condition. Again working backwards, we obtain the inverse implication, which completes the proof.
\end{proof}

\section{Tropical Hurwitz covers}
\label{sec:trop}
The idea of connecting Hurwitz numbers to tropical geometry, interpreting double Hurwitz numbers as the weighted count of the constructed tropical graphs was introduced in \cite{cavalieri2010tropical}. This tropical interpretation proved fruitful in producing polynomiality results; in particular it was used in the proof of wall-crossing formulae for double Hurwitz numbers in genus $0$.

\subsection{Tropical graphs}
\label{sec:tropgra}
Tropical geometry can be considered as a ``combinatorial shadow'' of algebraic geometry, where piece-wise linear objects called tropical graphs can be obtained as a skeleton of degenerated algebraic curves.
Though this tropicalisation procedure loses information, many properties of the algebraic curves continue to be determinable from their corresponding tropical curve. 

To begin, we define the edges of a tropical graph as follows.

\begin{definition}
    Let $\Gamma$ be a connected graph. We say an edge is an \textit{end} if it is adjacent to a $1$-valent vertex. Edges that are not ends are called \textit{bounded edges}.
\end{definition}

    We use $V(\Gamma)$ to denote the vertex set of the graph $\Gamma$. Furthermore, we denote the set of $1$-valent vertices (i.e. leaves) by $V_\infty(\Gamma)$, whereas we denote by $V_0(\Gamma)$ the set of vertices with a valency greater than $1$, called inner vertices. 
    
    Moreover, we use $E(\Gamma)$ to denote the edge set of the graph $\Gamma$, 
    $E_\infty(\Gamma)$ to denote the subset of ends, and $E_0(\Gamma)$ to denote the subset of bounded edges.
    Thus, we may define a tropical curve as follows.

    \begin{definition}
        An \textit{abstract tropical curve} is a connected graph $\Gamma$ (with $E(\Gamma) \neq \emptyset$) such that 
        \begin{enumerate}
            \item $\Gamma \backslash V_\infty (\Gamma)$ is a metric graph. That is, $\Gamma$ is equipped with a map
            $$ l : E(\Gamma) \rightarrow \mathbb{R} \cup \{\infty \}$$
            $$ e \mapsto l(e)$$
            such that $l(E(\Gamma) \backslash E_\infty(\Gamma) ) \subset \mathbb{R}$ and all ends $e \in E_\infty (\Gamma)$ have length $l(e) = \infty$,

            \item the inner vertices $v\in V_0(\Gamma)$ have non-negative integer weights. Namely, the graph $\Gamma$ is further equipped with a map
            $$ g: V_0(\Gamma) \rightarrow \mathbb{N}$$
            $$ v \mapsto g_v.$$

        \end{enumerate}
    \end{definition}

    If we consider a vertex $v\in V_0(\Gamma)$ we may define the integer $g_v$ to be the \textit{genus of} $v$. The genus of the curve $\Gamma$ is defined to satisfy $$ g(\Gamma) = \beta_1(\Gamma) + \sum_{v \in V_0(\Gamma)} g_v$$
    where $\beta_1(\Gamma)$ is the first Betti number of $\Gamma$.
    In our analysis, we only consider so-called ``explicit'' tropical curves $\Gamma$, which have $g_v = 0$ for ever inner vertex. That is, the genus of our graphs satisfies $$ g (\Gamma) = \beta_1(\Gamma).$$

    \begin{definition}\label{def combinatorial type of curve}
        The \textit{combinatorial type} of a tropical curve is the equivalence class of tropical curves where any two curves are equivalent if they differ only by the length of their edges.
    \end{definition}
    
    One tropical curve that is important in our story is the tropical projective line.

    \begin{example}
        We denote by $\mathbb{T}\mathbb{P}^{1}(\mathbb{C})$ the tropical $\mathbb{P}^{1}(\mathbb{C})$, and define this as $\mathbb{T}\mathbb{P}^{1}(\mathbb{C}) = \mathbb{R} \cup \{ \pm \infty \}$. Furthermore, we may construct vertices on this curve by picking a finite number of points $p_i \in \mathbb{R}$ and creating a vertex at each point. We define the length of the edges that joins two inner vertices to be their absolute distance in $\mathbb{R}$, and we define the length of the ends to be $\infty$.

        The genus of the tropical projective line is $g(\mathbb{T}\mathbb{P}^{1}(\mathbb{C})) = 0$ as expected-- \cref{fig:tropproj} shows us that this graph is a tree.
    \end{example}

    \begin{figure}[ht]
        \centering
        \includegraphics[width=\linewidth]{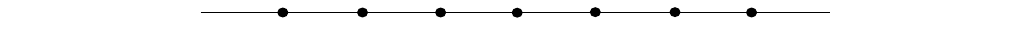}
        \caption{The tropical projective curve.}
        \label{fig:tropproj}
    \end{figure}

    \begin{remark}
        The above example is slightly ``cheating'', as a tropical curve with two $1$-valent ends, and all other vertices are $2$-valent.
    \end{remark}

    \subsection{Tropical covers}
    \label{sec:tropcov}
    There are many equivalent tropicalisations of Hurwitz numbers, including working via the symmetric group. In tropical geometry, algebraic curves tropicalise to combinatorial graphs. Therefore, maps between Riemann surfaces \textit{should} tropicalise to maps between graphs. 

    \begin{definition}\label{tropical cover def}
        A \textit{tropical cover} of two tropical curves $\Gamma_1, \Gamma_2 $ is a surjective map $f: \Gamma_1 \rightarrow \Gamma_2 $ which satisfies the following conditions:
        \begin{enumerate}
            \item (Locally integer affine linear) For any edge $e\in E(\Gamma_1)$, $f(e)$ is contained either in an edge of $\Gamma_2$ or in an inner vertex of $\Gamma_2$. Moreover, the map $f$ is piecewise integer affine linear. Namely, on an edge $e$ the map is locally a positive integer $w(e)$ called the weight of the edge, so that $$ l(f(e)) = w(e) l(e).$$

            \item (Balancing) For an inner vertex $v\in V_0(\Gamma_1)$, the local degree of $f$ at $v$ (denoted by $d_v$) is defined as follows; consider an edge $e_2$ adjacent to $f(v)$ in $\Gamma_2$, and sum the weights of all edges $e_2$ adjacent to $v$ in $\Gamma_1$ that map to $e_2$:
            $$ d_v = \sum_{e_1 \mapsto e_2} w(e_1).$$
        \end{enumerate} 
    \end{definition}

    The balancing condition thus ensures that the local degree of $f$ at $v$ is well defined, and independent of the choice of $e_2$.
    In our case, we choose the vertex sets $V(\Gamma_1)$ and $ V(\Gamma_2)$ in such a way that the tropical covers considered above are maps between graphs. In particular, the images and preimages of vertices are vertices.

    \begin{definition}
        A tropical cover between curves $\Gamma_1$ and $ \Gamma_2$ is called a \textit{tropical Hurwitz cover} if it satisfies the \textit{local Riemann--Hurwitz condition} at every vertex $v\in \Gamma_1$, meaning that if $v \mapsto v_2$ with local degree $d_v$, we have
        \begin{equation}\label{local RH condition eq}
            0 \leq d_v(2-2g(v_2)) - \sum_{e \text{ adj. to }v} \Big(w(e) -1 \Big) - (2-2g(v)).
        \end{equation} 
    \end{definition}
    
    The right hand side of this inequality can be thought of as a measure of the ramification at the vertex $v$, which cannot be negative.
    When considering explicit tropical curves, this expression simplifies to 
        \begin{equation}\label{local RH condition eq explicit}
            0 \leq 2d_v - \sum_{e \text{ adj. to }v} \Big(w(e) -1 \Big) - 2.
        \end{equation} 

     \begin{definition}
        The \textit{degree} $d$ of a tropical cover is the sum over all local degrees of preimages of a point $y\in \Gamma_2$. That is, we consider all points $x\in \Gamma_1$ such that $x\mapsto y$ and obtain $$ d = \sum_{x\mapsto y} d_x.$$
    \end{definition}

    By the balancing condition (\cref{tropical cover def} (2)), the degree of the tropical cover is independent of the choice of $y\in \Gamma_2$ that we take.

    Let us consider an end $e\in E_\infty(\Gamma_2)$, and let $\mu_e \vdash d$ be the partition obtained from the weights $w(e_1)$ of the ends $e_1\in E_\infty(\Gamma_1)$ which are mapped by $f$ onto $e$.

    \begin{definition}
        The partition $\mu_e$ is called the \textit{ramification profile} above $e$
    \end{definition}

    \begin{definition}
        We define the \textit{combinatorial type of a tropical cover} to be the equivalence class of covers when we drop all metric information about the curves. 
        Namely, it is comprised of the combinatorial types of the tropical curves $\Gamma_1$ and $ \Gamma_2$, the weights of the map, and the information of which edges map to which.   
    \end{definition}

        Given a tropical cover $f:\Gamma_1 \rightarrow \Gamma_2$, the combinatorial type of the cover, and the lengths of $\Gamma_2$, it is possible to recover the metric on $\Gamma_1$. Indeed, let us consider any edge $e_1$ of $\Gamma_1$ with weight $w(e_1)$ and unknown length $l(e_1)$. If $f(e_1)=e_2$, then we have $l(e_2) = w(e_1)\cdot l(e_1)$ by \cref{tropical cover def} $(1)$.

    \begin{remark}\label{realizability}
    It turns out that the local Riemann--Hurwitz condition is in fact a \textit{realizability} condition. Namely, only tropical curves satisfying \cref{local RH condition eq} at every point can be degenerations of covers of algebraic curves.
    \end{remark}
    
Now, we wish to define morphisms of tropical curves.
    
    \begin{definition}\label{def isomorphic covers}
        Tropical covers $f_1: \Gamma_1 \rightarrow \tilde{\Gamma}$ and $f_2 : \Gamma_2 \rightarrow \tilde{\Gamma}$ are called \textit{isomorphic} if there exists an isomorphism $ \phi: \Gamma_1 \rightarrow \Gamma_2  $ of the underlying tropical curves $\Gamma_1, \Gamma_2$ such that $ f_2 \circ \phi = f_1 $.
        \end{definition}

    In particular, we identify curves for which the following diagram commutes. 
    $$\begin{tikzcd}
    \Gamma_1 \arrow[rd, "f_1" ' near start] \arrow[rr, "\phi" ] & &  \Gamma_2 \arrow[ld, "f_2" near start] \\
     & \tilde{\Gamma} & 
    \end{tikzcd}$$ 
    Isomorphism classes of covers are an equivalence relation, allowing us to consider representatives for certain isomorphism classes of graphs. In our discussion this amounts to considering covers as equivalent whenever the maps sends the same vertices in $\Gamma_1$ to the same vertices in $\tilde{\Gamma}$, though when edges $e_i$ adjacent to a vertex $v$ in $\Gamma_1$ are mapped to an edge $\tilde{e}$ in $\tilde{\Gamma}$, the weights may be permuted among the edges $e_i$. 

     There are some shapes of special interest that appear in our tropical monodromy graphs, as they result in automorphisms of the tropical graphs.
  \begin{figure}[ht]
     \centering
     \includegraphics[width=\linewidth]{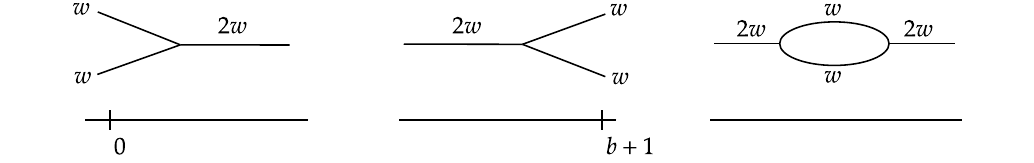}
     \caption{Balanced left pointing fork, balanced right pointing fork, and balanced wiener.}
     \label{fig:forks n weiners}
 \end{figure}
 
 \begin{definition}
 A \textit{left balanced pointing fork} (resp. \textit{right}) is a tripod with weights $w, w, 2w$ such that the edges of weight $w$ lie over $0$ (resp. $b+1$) (see \cref{fig:forks n weiners}).
 \end{definition}
 
 \begin{definition}\label{def wiener}
 A \textit{balancing wiener} appears in the graph when a strand of weight $2w$ splits into two strands of weight $w$, and then rejoins into a strand of weight $2w$ (see \cref{fig:forks n weiners})
 \end{definition}

 \subsection{Tropical double Hurwitz numbers}
 \label{sec:tropdhn}
The tropical monodromy graphs of \cite{cavalieri2010tropical} arise using an analysis of cut-and-join procedures in the symmetric group $S_d$. Fixing integers $g\geq 0$ and $d>0$, taking $\mu, \nu$ to be partitions of $d$, and taking $b=2g-2+\ell(\mu)+\ell(\nu)>0$, they are constructed as follows.

 \begin{definition}\label{def monodromy graph}
 \textit{Monodromy graphs of type} $(g, \mu, \nu)$ are tropical graphs with a map projecting to the segment $[0 , b+1] \subset \tPC$, constructed as follows:
 \begin{enumerate}
     \item Begin with $m$ strands over $0$ that are labelled by $\mu_1, \ldots, \mu_m$. These $\mu_i$ are called the weight of their respective strands.
     
     \item Create a $3$-valent vertex over the point $1$ by either joining two strands or cutting one strand that has weight strictly greater than $1$.
        \begin{itemize}
            \item For a join, the new strand is weighted with the sum of the weights of the strands joined,
            
            \item For a cut, the new strands are weighted in all possible positive ways that sum to the weight of the cut strand.
        \end{itemize}
    
    \item Consider one representative for each isomorphism class of tropical covers (as in \cref{def isomorphic covers}) 
    
    \item Repeat steps $2, 3$ above for each consecutive integer up to $b$.
    
    \item Consider all connected graphs that conclude over $b+1$ with $n$ strands of weight $\nu_1, \ldots, \nu_n$.
     
 \end{enumerate}
 
 \end{definition}

 \begin{remark}
 These tropical monodromy graphs should be treated as abstract graphs with weighted edges, mapping to the segment $[0, b+1]$ of the tropical projective line. That is to say, the relative positions of the strands is inconsequential, and there are no crossings between the strands.
 \end{remark}

 \begin{remark}
    These monodromy graphs can be considered as graphs with half edges. The vertex set $V(G)$ of a monodromy graph $G$ consists of the $b$ many $3$-valent vertices, the edge set $E(G)$ consists of the inner edges between these $b$ vertices, and we can consider the set of half edges $E^\prime$ to consist of the ends, i.e. the unbounded rays over $(-\infty, 1)$ and $(b+1, \infty)$ labelled with the parts of $\mu$ and $\nu$ respectively.
 \end{remark}

Let us illustrate this construction with an example. 

\begin{figure}[ht]
     \centering
     \includegraphics[width=\linewidth]{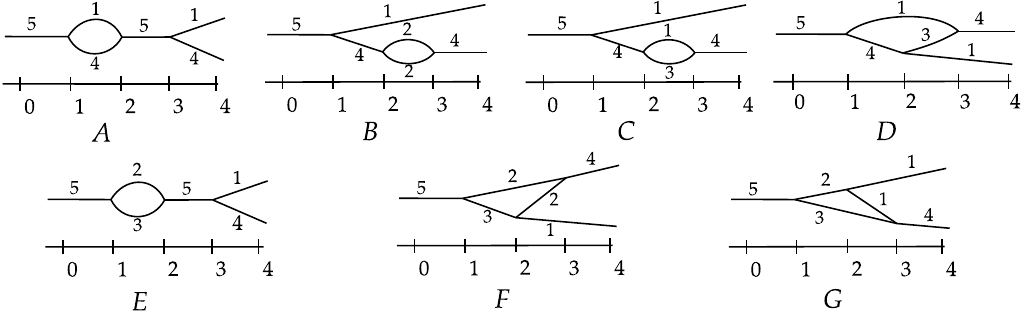}
     \caption{Monodromy graphs of type $(1, (5), (4, 1) )$.}
     \label{fig:monogrex2}
 \end{figure}
 
    \begin{example}\label{ex monogr2}
    We consider monodromy graphs of type $(1, (5), (4, 1) )$. That is, we construct connected graphs with $1$ strand of weight $5$ above $0$, and with $2$ strands of weights $1$ and $4$ above $4$. Furthermore, these graphs have $b=2-2+1+2=3$ many $3$-valent vertices over the points $1, 2$, and $ 3$ where strands may be joined, or cut into two strands of positive weight. 
    
    The graphs of this form are depicted in  \cref{fig:monogrex2}.
    We note that the implicit map for these graphs is the projection to the segment $[0, 4] \subset \tPC $. We further note that cover $B$ is the only cover containing an automorphism, namely a balanced wiener (as defined in \cref{def wiener}).
    \end{example}

    Using these tropical monodromy graphs, we may define tropical double Hurwitz numbers as follows. 

    \begin{definition}\label{deftropdhn}
    Let $g\geq 0, d>0$ be integers, $\mu, \nu$ partitions of $d$, and $b=2g-2+\ell(\mu) +\ell(\nu) >0$. The \textit{tropical double Hurwitz number} $H_g^{\text{trop}}(\mu, \nu)$ is the weighted sum of monodromy graphs $\Gamma$ of type $(g, \mu, \nu)$ by the formula 
    $$ H_g\tr(\mu, \nu) = \sum_{\Gamma} \frac{|\Aut(\mu)| |\Aut(\nu)|}{|\Aut (\Gamma )|} \prod_{e \text{ inner edge}} w(e) $$
    taking the product over the interior edge weights of the monodromy graphs $\Gamma$, where factors of $1/2$ for each balancing fork and wiener amounts to the size of the automorphism group of \ $\Gamma$.
    \end{definition}

    Using these tropical double Hurwitz numbers we may enumerate classical double Hurwitz numbers as a weighted sum of tropical covers using the following tropical correspondence theorem.

    \begin{theorem}\label{trop corr thm for DHN}(\cite[Section 4]{cavalieri2010tropical})
    Fix integers $d>0$ and $g\geq 0$, and let $\mu$, $\nu$ be partitions of $d$. Then, the classical count of the double Hurwitz number is equal to the tropical count, i.e.
    $$H_g(\mu, \nu) = H_g\tr(\mu, \nu). $$
    \end{theorem}

    We give a demonstration of this count as follows.

    \begin{example}
    We consider the monodromy graphs of type $(1, (5), (4, 1))$ constructed in \cref{ex monogr2} (which can be seen in \cref{fig:monogrex2}).
  
    Let us note the following information:
    \begin{itemize}
        \item The graph $B$ contains a balanced wiener, meaning that $|\Aut(B)| = 2$
        
        \item None of the other graphs contain balanced forks or balanced wieners, so $|\Aut(\Gamma)| = 1$ for $\Gamma = A, C, D, E, F, G$.

        \item The partitions $\mu = (5),\ \nu = (4, 1)$ contain no automorphisms, giving $|\Aut(\mu)| = 1 = |\Aut(\nu)|$.
        
    \end{itemize}

    \begin{table}[ht]
    \begin{tabular}{|c| c c | c |} 
    \hline 
    Graph $\Gamma$ & $\frac{|\Aut(\mu)||\Aut(\nu)|}{|\Aut (\Gamma )|}$ & $ \prod_{\hat{e} } w(\hat{e})$ & Total \\ [0.5ex] 
    \hline\hline
    \textrm{A} & $1$ & $ 4 \cdot 5$ & $20$ \\ 
 
    \hline
    \textrm{B} & $1/2$ & $4\cdot 2 \cdot 2$ & $8$ \\ 
 
    \hline
    \textrm{C} & $1$ & $4\cdot 3$ & $12$ \\ 

    \hline
    \textrm{D} & $1$ & $4\cdot 3$ & $12$ \\ 
 
    \hline
    \textrm{E} & $1$ & $2\cdot 3 \cdot 5$ & $30$ \\ 
 
    \hline
    \textrm{F} & $1$ & $3\cdot 2 \cdot 2$ & $12$ \\ 
 
    \hline
    \textrm{G} & $1$ & $2\cdot 3$ & $6$\\ [1ex] 
    \hline
    \end{tabular}  
    \caption{\label{table2DHN}The weight for each pruned monodromy graph of type $(2, (1, 1, 1), (3))$.}
    \end{table}
        
    
    By \cref{deftropdhn} and \cref{trop corr thm for DHN}, $H_1((5), (4,1))$ is found by summing over the total column in \cref{table2DHN} to yield
    $$H_1^{\textrm{trop}} ((5), (4,1)) =  100 = H_1((5), (4, 1)).$$    
    
    \end{example}

\section{Tropical pruned Hurwitz numbers}
\label{sec:troppru}
In this section we construct a new interpretation of pruned double Hurwitz numbers $PH_g(\mu, \nu)$ using tropical covers via the cut-and-join recursion for $PH_g(\mu, \nu)$. Based on this new interpretation, we study the polynomial structure on pruned double Hurwitz numbers from a tropical perspective in \cref{sec:polytrop}.

\subsection{Pruned monodromy graphs}
To state the aforementioned recursion, we first fix $g\geq 0$ and $d>0$, consider $\mu, \nu$ partitions of $d$, and denote by $b = 2g-2 + \ell(\mu) + \ell(\nu)$.

\begin{theorem}\label{thm pruned recursion}(\cite[Theorem 24]{zbMATH06791415})
    For $b >0$ and $(g, \ell(\nu))\notin \{ (0, 1), (0, 2)\}$, the pruned double Hurwitz number $PH_g(\mu, \nu)$ satisfies the following recursion
    \begin{multline}
       PH_g(\mu, \nu)  =  \sum_{i<j} \sum_{I \subset \{1, \ldots, \ell(\mu)\} } \ \ \sum_{\mathclap{\substack{\alpha + |\mu_{I^c}| \\ 
	                          = \nu_i + \nu_j } }} \  \alpha \cdot \frac{(b-1)!}{(b-(|I^c|+1))!} \cdot (|I^c| +1)! \cdot \prod_{s\in \mu_{I^c}} s  \cdot  PH_g(\mu_I, (\nu \backslash\{\nu_i, \nu_j\}, \alpha) )  \\
   + \frac{1}{2}  \sum_{i=1}^{\ell(\nu)} \sum_{I \subset \{1, \ldots, \ell(\mu)\} } \ \ \ \ \sum_{\mathclap{\substack{\alpha + \beta + |\mu_{I^c}| \\ 
	                           = \nu_i } }} \ \ \alpha  
\cdot \beta \cdot \frac{(b-1)!}{(b-(|I^c|+1))!} \cdot (|I^c| +1) \cdot \prod_{s\in \mu_{I^c}} s \cdot    PH_{g-1}(\mu_I, (\nu \backslash \{\nu_i \} , \alpha, \beta ) )   \\
 + \frac{1}{2} \sum_{i=1}^{\ell(\nu)} \sum_{g_1 + g_2 = g}^{\text{stable}}  \quad \sum_{\mathclap{\substack{\nu_{J_1} \coprod \nu_{J_2} = \\
                        \nu \backslash \{\nu_i\}  } }} \qquad
 \sum_{\mathclap{\substack{I_1, I_2 \subset\\
                            \{ 1, \ldots, \ell(\mu ) \}\\
                            \text{disjoint}  } }} \qquad \
\sum_{\mathclap{\substack{\alpha + \beta +\\
                        |(\mu_{I_1} + \mu_{I_2})^c|\\
                        = \nu_i  } }} \quad 
\alpha \cdot \beta \cdot \frac{(b-1)!}{b_1! \cdot b_2!} \cdot (|(I_1 +I_2)^c| +1)! \cdot \\
\prod_{s\in \mu_{(I_1 \cup I_2)^c}} s \cdot  PH_{g_2} (\mu_{I_2}, (\nu_{J_2}, \beta))  .
    \end{multline}
    where $b_1 = 2g_1 -2 + |I_1| + |J_1| +1 , \ b_2 = 2g_2 -2 + |I_2| + |J_2| +1$, and the ``stable'' terms in the sum mean that we exclude terms with $(g_i, |J_i|) \in \{ (0, 1), (0, 2)\} $.
    \end{theorem}

    Using this recursion we introduce a new ``pruned'' monodromy graph structure. These graphs may additionally feature coloured ends and $n$-valent vertices, which we formalise as follows. 

    \begin{definition}
        The coloured ends of a graph $\Gamma$ is a proper subset $CE(\Gamma) \subsetneq E_\infty(\Gamma)$ of the set of ends of the graph. 
    \end{definition}
    
    \begin{definition}
        A regular edge $e$ of a graph $\Gamma$ is an edge (or end) that is not a coloured end. In particular, $e\in E(\Gamma) \backslash CE(\Gamma)$.
    \end{definition}
    
    When drawn, we represent these edges as expected; regular edges are drawn as a plain black edge and coloured ends are drawn as red dashed lines in our figures.

    \begin{definition}
    A regular $n$-valent vertex is a vertex that is joined by $n$ regular strands.
    \end{definition}

    We specify certain types of vertices as follows.
    
    \begin{definition}
    Initial vertices are regular $n$-valent vertices joined by $n-2$ ends on the left, which split into two regular edges on the right.
    \end{definition}

    \begin{definition}
    A secondary vertex is a regular $3$-valent vertex that is either: joined by two regular inner edges on the left that join into one regular edge on the right, or it is joined by one regular inner edge on the left that splits into two regular edges on the right.
    \end{definition} 

    We construct our pruned monodromy graphs such that each strand of our graph must first be joined to an initial vertex. These initial vertices create new automorphism shapes in our graph, that we call $k$-pronged forks.

    \begin{definition}
        An $k$\textit{-pronged fork} appears in a graph when $k$ strands of weight $\mu_1, \ldots, \mu_k$ lying over $0$ join at an initial vertex.
    \end{definition}

    We take the automorphism factor of a $k$--pronged fork to be $|\Aut(\mu_1, \ldots, \mu_k)|$, the automorphism of the partition $(\mu_1, \ldots, \mu_k)$.

    When constructing our new curves and covers, we wish to define a new equivalence relation between the tropical covers. In this instance, we do not care about the ordering of vertices of disjoint connected components before the components join together. We make this precise as follows.

    \begin{definition}\label{def quasiisomorphic}
        We call two tropical covers $\pi_1 : \Gamma_1 \rightarrow \tPC, \pi_2 : \Gamma_2 \rightarrow \tPC$ of the tropical projective line \textit{quasi-isomorphic} if: 
        \begin{itemize}
            \item the underlying graphs $\Gamma_1, \Gamma_2$ are isomorphic 
            with isomorphism $\phi : \Gamma_1 \rightarrow \Gamma_2$, where coloured edges map to coloured edges, 

            \item the covers carry the same edge weights $w(e) = w(\phi(e))$ for an edge $e\in \Gamma_1$,

            \item the covers carry the same vertex multiplicities $m(v) = m(\phi(v))$ for a vertex $v\in \Gamma_1$,

            \item and furthermore the edge $\phi(e)$ carries the same direction as $e$. That is, if $e$ is joined to vertices above the point $i$ and the point $j$ where $i<j$, then $\phi(e)$ is mapped to vertices $\phi(i)$ and $\phi(j)$ where $\phi(i) < \phi(j)$.
        \end{itemize}
    \end{definition}

    Intuitively, quasi-isomorphic covers are covers of graphs with the same structure and the same weights, though if we consider any partial graph from the left, vertices that originate in disjoint connected components may project to $\tPC$ in a different order up until the components connect (see \cref{fig:quasi}). 

    \begin{figure}
        \centering
        \includegraphics[width=\linewidth]{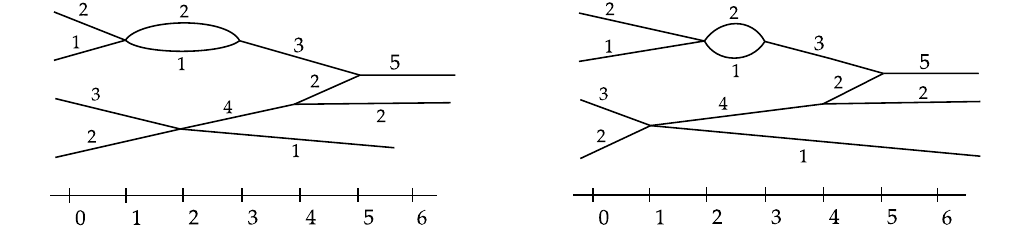}
        \caption{Isomorphic graphs with quasi-isomorphic covers projecting to the segment $[0, 6]\subset \tPC$.}
        \label{fig:quasi}
    \end{figure}

    Now, we wish to formulate a tropical interpretation of pruned double Hurwitz numbers using the recursion as stated in \cref{thm pruned recursion}. 

    \begin{construction}\label{construction prunes}
       Let $g\geq 0$ and $d> 0$ be integers, and let $\mu = (\mu_1, \ldots , \mu_m)$ and $\nu = (\nu_1, \ldots, \nu_n) $ be partitions of $d$. Let us consider $b = 2g-2 + m + n > 0 $ and $(g, \ell(\nu)) \notin \{ (0, 1), (0,2) \}$. We associate to this data a tropical graph $\Gamma$ and a map to the segment $[0, \tilde{b} + 1]$ as follows:
        \begin{enumerate}
            \item We start with $m$ strands over $0$ that are labelled by $\mu_1, \ldots, \mu_m$, where $\mu_k$ is the weight of the respective strands. We choose a proper subset $I\subsetneq [m]$ of these strands to be coloured ends, keeping the remaining strands $\mu_j$,  $j \in [m ] \backslash I $, as regular strands. 

            \item The recursion tells us the specific ways in which we may cut regular edges and join regular or coloured edges at any given vertex. Over the first vertex, we deal with the base case where $(g, \ell(\nu)) $ is locally equal to $(0, 2)$. We consider the other base cases in \cref{rem base case prunedtrop}.

            \item Over the point $1$ we construct a regular $n$-valent vertex, where $n\geq3$. To do this, we join together $n-2$ regular edges that originate over $0$ 
            to the vertex lying above $1$, which we then split into two outgoing regular edges on the right. We do not join any coloured edges to the vertex above $1$. This vertex is an initial vertex.

            \item This process of creating initial vertices is repeated at each successive integer until every regular strand originating above $0$ has been joined into an $n$-valent vertex, where $n\geq 3$.
            In doing this, we only consider attaching regular strands from the left that are original regular edges $\mu_j,\ j\in [m]\backslash I$ originating above $0$. In particular, we do not join any inner edges or coloured edges of our graph to these initial vertices. 

            \item Upon attaching every regular strand $\mu_j, \ j\in [m]\backslash I$, to an initial vertex, we may begin to construct secondary vertices. If we denote by $s$ the number of secondary vertices, we create $s = b - \ell(\mu)$ secondary vertices. Let us denote by $i$ the number of initial vertices of $\Gamma$ that we have constructed. Thus, we find that $\Tilde{b} = i+s$, so we are projecting to the segment $[0, i+s+1]$.

            \item Over subsequent vertices, the recursion tells us the manner in which we can construct these secondary vertices. That is, for the vertex above $i+1$ we may attach edges in two ways. Namely, 
                    we may create a regular $3$-valent vertex by either:
                    \begin{itemize}
                        \item joining two regular strands at the vertex, with one new regular edge coming out of the vertex on the right, 

                        \item or cutting one regular strand, which we split into two new regular edges exiting the vertex to the right. 
                    \end{itemize} 
                    When we have appropriately constructed our $3$-valent secondary vertex, we may join a non-negative number of coloured ends $\mu_k$, $k\in I$, to this same vertex. 
            
    
            \item Consider one representative for each quasi-isomorphism class of labelled graphs. 
    
            \item We repeat steps $(6)$ and $(7)$ for each consecutive integer up to $\Tilde{b}$. 

            \item When we reach vertex $\tilde{b}$, we consider connected graphs that terminate with $\ell(\nu) = n$ edges above $\Tilde{b} +1$. 
    
     
        \end{enumerate}
        We thus obtain a graph $\Gamma$ with a map projecting to the interval $[0, \Tilde{b}+1]$. We call this graph (along with the projection map) a pruned monodromy graph of type $(g, \mu, \nu)_{\text{P}}$.
    \end{construction}

    \begin{remark}\label{rem base case prunedtrop}
        The base cases that are not covered by the construction are those with $$(g, \ell(\nu)) \in \{ (0, 1), (0, 2) \}.$$  
         We construct monodromy graphs of type $(g, \mu, \nu)_{\text{P}}$ in order to express tropical pruned double Hurwitz number $PH\tr_g(\mu, \nu)$ as a weighted count of tropical graphs. However, we know by \cref{base case pruned ex}, that $PH_0((\mu_1, \ldots, \mu_m) , (\nu_1)) = 0$ for $m\in \{1,2,  \ldots\}$.  Thus, we represent these cases by the tropical graphs as illustrated in \cref{fig:basecaseprunes} (a) and (b). 
         
         When computing the weighted count of these graphs we prescribe these values to be exactly 
         the classical pruned double Hurwitz count. That is, we say that in graph (b) corresponding to $PH\tr_0( (\mu_1, \ldots, \mu_m) , (\nu_1) )$, the single vertex has multiplicity $m(v) = 0$. By convention, the pruned monodromy graph (a) on zero vertices has weight $0$. Moreover, we prescribe the multiplicity of the single vertex in graph (c) to be the pruned double Hurwitz number $PH_0((\mu_1, \ldots, \mu_m), (\nu_1, \nu_2))$. 

        Furthermore, any graphs that result from those of \cref{fig:basecaseprunes} (a) and (b) have a vertex multiplicity factors of $0$. Thus, in our construction we \textit{could} have allowed the vertex above $1$ to be joined by two (or more) regular strands on the left that combine into one regular strand on the right, but this is essentially a redundant case -- when we take the weighted sum over these graph, their weight is $0$ from the first vertex.
    \end{remark}

    \begin{figure}[ht]
        \centering
        \includegraphics[width=\linewidth]{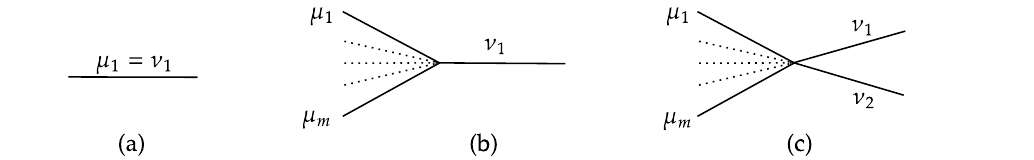}
        \caption{The tropical picture for $PH_0( (\mu_1), (\nu_1))$, $PH_0((\mu_1, \ldots, \mu_m), (\nu_1))$, and $PH_0((\mu_1, \ldots, \mu_m), (\nu_1, \nu_2))$.}
        \label{fig:basecaseprunes}
    \end{figure}


    \begin{definition}
    \label{def:prunedmonod}
        A pruned monodromy graph of type $(g, \mu, \nu)_{\text{P}}$, with $b=2g - 2 +\ell(\mu) +\ell(\nu)$, is a graph $\Gamma$ projecting to the segment $[0, \tilde{b} +1]$ with the following properties:
        \begin{enumerate}
            \item The graph $\Gamma$ has $c$ many coloured edges, where $0 \leq c<\ell(\mu)$.

            \item The graph $\Gamma$ has $s = b - \ell(\mu)$ many secondary vertices. Furthermore, if $\Gamma$ has $i$  initial vertices, then taking $\Tilde{b} = i+s$, $\Gamma$ projects to the segment $[0, s + i + 1]$.

            \item The genus of $\Gamma$ is $g$.

            \item We assign a weight to each edge $e$. This weight $w(e)$ is a positive integer.

            \item At each inner vertex 
            the sum of the weights of the incoming edges (both regular and coloured) is the sum of the outgoing weights of the edges. 

            \item The weights of the $n$ strands above $\Tilde{b} + 1 $ are $\nu_1, \ldots, \nu_n$.
            
        \end{enumerate}
        We denote the set of pruned monodromy graphs of type $(g,\mu,\nu)\p$ by $\mathfrak{TP}_g(\mu,\nu)$.
    \end{definition}

\subsection{Tropical pruned double Hurwitz numbers}

    As in the case of tropical double Hurwitz numbers, we wish to define pruned double Hurwitz numbers as the weighted sum over tropical graphs.

    \begin{definition}\label{pruned weights}
        Tropical pruned double Hurwitz numbers $PH\tr_g (\mu, \nu)$ are the weighted sum of pruned monodromy graphs $\Gamma$ of type $(g, \mu, \nu)\p$, 
        $$ PH\tr_g (\mu, \nu) = \sum_{\Gamma\in \mathfrak{TP}_g(\mu,\nu)} \frac{|\Aut(\mu)||\Aut(\nu)|}{|\Aut(\Gamma)|}\prod_{e \text{ inner edge}} w(e) \prod_{\hat{e} \text{ col. edge}} w(\hat{e}) \prod_{v \text{ vertex}} m(v)$$
        where the automorphism group of $\Gamma$ is made up of the factors of $1/2$ for each right balanced fork and wiener and the factor of $1/|\Aut(\mu_1, \ldots, \mu_k)|$ for any $k$--pronged fork of weights $\mu_1, \ldots, \mu_k$.
        Furthermore, the product over the edge weights ranges over the inner edges in the first product and over the coloured edges in the second product. Moreover, the multiplicity of each vertex is prescribed as follows:
        \begin{itemize}
            \item If $(g, \ell(\nu)) = (0, 1)$, then the graph looks like \cref{fig:basecaseprunes} (a) or (b). In case (a), we define the product of the vertex multiplicities over zero vertices to be $0$.
            Likewise, if we consider the vertex $v$ in case (b), where we have a regular $n$-valent vertex with $n-1$ edges joining from the left into $1$ on the right, we prescribe this to have multiplicity $$m(v)=0.$$
            
            \item The vertex multiplicity for a initial vertex $v$, i.e. a regular $n$-valent vertex as depicted in \cref{fig:basecaseprunes} (c), is the local pruned double Hurwitz number; $$m(v) = PH_0( (\mu_1, \ldots, \mu_{n-2}), (\nu_1, \nu_2)).$$ 
            
            \item For a given secondary vertex $v$ let us denote by $|c_v|$ the number of coloured ends joining $v$, and denote by $|\Tilde{c}|$ the number of coloured ends joined to vertices coming before $v$ in the connected component that $v$ lies in. Furthermore, we denote by $s$ the number of secondary vertices that come before $v$ in the connected component of the graph that $v$ lies in, while we denote by $v_i$ the 
            initial vertices in the same connected component as $v$, that come before $v$. 
            
            \item[cut:] If the secondary vertex $v$ consists of a regular strand being cut into two regular strands (possibly with $|c_v|$ coloured edges joining into $v$), then the multiplicity $m(v)$ of $v$ is 
            $$m(v) = \frac{\Big( \sum_{v_i} \big(\val(v_i) -2\big) + s + |\Tilde{c}| +|c_v|\Big)!}{\Big(\sum_{v_i} \big(\val(v_i) -2\big) + s + |\Tilde{c}|\Big)!} \cdot \big( |c_v| +1\big)!$$
            In particular $m(v)=1$ if no coloured edges join $v$.

            \item[c. join:] If two regular strands from the same connected component of the graph before $v$ are joined together at the secondary vertex $v$, then the multiplicity of $v$ is
            $$ m(v) = \frac{\Big(\sum_{v_i} \big(\val(v_i) -2\big) + s + |\Tilde{c}| +|c_v|\Big)!}{\Big(\sum_{v_i} \big(\val(v_i) -2\big) + s + |\Tilde{c}|\Big)!} \cdot \big(|c_v| +1\big).$$
            As in the above case, $m(v)=1$ if no coloured edges join $v$.

            \item[d. join:] If two regular strands from two disjoint connected components join at the inner vertex $v$, there are two cases that we distinguish between; \textit{degenerate disconnected joins} and \textit{non-degenerate disconnected joins}. 
            \begin{itemize}
                \item Degenerate disconnected joins are the cases in our recursion that do not qualify as ``stable'', which arise when $(g, \ell(\nu)) \in \{(0, 1), (0 ,2)\}$.
                In particular, if one (or both) of the strands joining at $v$ comes from a connected component that is a graph of genus $g=0$ (i.e. a tree) up to this vertex $v$ \textit{and} furthermore has only one choice or two choices of strands to join at vertex $v$, then $v$ has degenerate multiplicity:
                $$ m(v) = 0.$$

                \item A non-degenerate disconnected join occurs in the case in which the strands joining at $v$ come from connected components $ \Gamma_1, \Gamma_2$ that up to now; are not trees, and/or have more than three strands above $v$ that we may choose to join at $v$. 
                In this case, we denote by $s_j$ and $|\tilde{c_j}|$ the number of secondary vertices and the number of coloured edges respectively in each connected component $\Gamma_j$ before the vertex $v$, $j\in \{1, 2\}$. Furthermore, we denote by $v_{i_j}$ the initial vertices in the connected component $\Gamma_j$, $j \in \{1, 2 \}$ coming before the vertex $v$. Then the multiplicity of $v$ is 
            $$m(v) = \frac{\Big(\sum_{v_{i_1}} \big(\val(v_{i_1}) -2\big) + s_1 +\sum_{v_{i_2}} \big(\val(v_{i_2}) -2\big) + s_2  + |\Tilde{c_1}| + |\Tilde{c_2}| +|c_v|\Big)!}{\Big(\sum_{v_{i_1}} \big(\val(v_{i_1}) -2\big) + s_1 + |\Tilde{c_1}|\Big)!\cdot  \Big(\sum_{v_{i_2}} \big(\val(v_{i_2}) -2\big) + s_2 + |\Tilde{c_2}|\Big)!} \cdot \big(|c_v| +1\big)!$$
            \end{itemize}

        \end{itemize}
    \end{definition}

    \begin{remark}
        In \cref{realizability}, the ``realisability'' condition of tropical curves was considered. We may note that the pruned tropical graphs that we construct are not, what we would call, the ``natural'' tropicalisation. The covers that we establish are a purely combinatorial way of giving a geometric meaning to a combinatorial count.
    \end{remark}

    We now state our tropical correspondence theorem for pruned double Hurwitz numbers as follows.

    \begin{theorem}\label{corr thm pdhn}(Correspondence theorem for tropical pruned double Hurwitz numbers)
        Let us fix integers $d>0$ and $g\geq 0$, and let $\mu, \nu$ be partitions of $d$. Then, the classical count of the pruned double Hurwitz numbers is equal to the tropical count. Namely, $$ PH_g(\mu, \nu) = PH\tr_g(\mu, \nu).$$
    \end{theorem}

    \begin{proof} We prove this correspondence by induction.
    
        To begin, we consider the base cases. By definition these are equal to their classical pruned double Hurwitz analogue. It remains to show that tropical pruned double Hurwitz numbers satisfy the same recursion as the classical count. 
        In particular, let us assume that $$ PH\tr_{g^\prime}(\mu^\prime, \nu^\prime) = PH_{g^\prime}(\mu^\prime, \nu^\prime)$$
        for all $g^\prime, \mu^\prime, \nu^\prime$ such that $2g^\prime -2 + \ell(\mu^\prime) + \ell(\nu^\prime) = k \geq 1$. Now, we claim that $$PH\tr_g(\mu, \nu) = PH_g(\mu, \nu)$$
        for $g, \mu, \nu$ such that $2g-2 + \ell(\mu) + \ell(\nu) = k+1 $ (excluding the base cases $(g, \ell(\nu)) = \{(0, 1), (0, 2)\}$).
        To show this equality, we first recall the recursion for pruned double Hurwitz numbers as in \cref{thm pruned recursion}. We split this into three distinct parts $PH_g(\mu, \nu)^\text{cut}, PH_g(\mu, \nu)^\text{cj}, PH_g(\mu, \nu)^\text{dj} $
        \begin{itemize}
            \item Let us consider the ``cut'' case in the recursion, and denote by $PH_g(\mu, \nu)^\text{cut}$ the first sum:
    \begin{multline}\label{pruned c recur}
       PH_g(\mu, \nu)^\text{cut}  =  \sum_{i<j} \sum_{I \subset \{1, \ldots, \ell(\mu)\} } \ \ \sum_{\mathclap{\substack{\alpha + |\mu_{I^c}| \\ 
	                          = \nu_i + \nu_j } }} \  \alpha \cdot \frac{(b-1)!}{(b-(|I^c|+1))!} \cdot (|I^c| +1)! \cdot \prod_{s\in \mu_{I^c}} s \cdot     PH_g(\mu_I, (\nu \backslash\{\nu_i, \nu_j\}, \alpha) )     .
    \end{multline}
            
    \item Let us consider the ``connected join'' case in the recursion, denoting by $PH_g(\mu, \nu)^\text{cj}$ the second sum:
        \begin{multline}\label{pruned cj recur}
       PH_g(\mu, \nu)^\text{cj}  =   \frac{1}{2}  \sum_{i=1}^{\ell(\nu)} \sum_{I \subset \{1, \ldots, \ell(\mu)\} } \ \ \ \ \sum_{\mathclap{\substack{\alpha + \beta + |\mu_{I^c}| \\ 
	                           = \nu_i } }} \ \ \alpha  
\cdot \beta \cdot \frac{(b-1)!}{(b-(|I^c|+1))!} \cdot (|I^c| +1) \cdot \prod_{s\in \mu_{I^c}} s \cdot PH_{g-1}(\mu_I, (\nu \backslash \{\nu_i \} , \alpha, \beta ) ).
    \end{multline}
    
    \item Let us consider the ``disconnected join'' case in the recursion, and denote the third sum by $PH_g(\mu, \nu)^\text{dj}$:
        \begin{multline}\label{pruned dj recur}
       PH_g(\mu, \nu)^\text{cj}  = \frac{1}{2} \sum_{i=1}^{\ell(\nu)} \sum_{g_1 + g_2 = g}^{\text{stable}}  \quad \sum_{\mathclap{\substack{\nu_{J_1} \coprod \nu_{J_2} = \\
                        \nu \backslash \{\nu_i\}  } }} \qquad
 \sum_{\mathclap{\substack{I_1, I_2 \subset\\
                            \{ 1, \ldots, \ell(\mu ) \}\\
                            \text{disjoint}  } }} \qquad \
\sum_{\mathclap{\substack{\alpha + \beta +\\
                        |(\mu_{I_1} + \mu_{I_2})^c|\\
                        = \nu_i  } }} \quad 
\alpha \cdot \beta \cdot \frac{(b-1)!}{b_1! \cdot b_2!} \cdot (|(I_1 +I_2)^c| +1)! \cdot \\
\prod_{s\in \mu_{(I_1 \cup I_2)^c}} s \cdot PH_{g_1}(\mu_{I_1}, (\nu_{J_1}, \alpha)) \cdot  PH_{g_2} (\mu_{I_2}, (\nu_{J_2}, \beta))  .
    \end{multline}

        \end{itemize}
        Putting this together, we have $PH_g(\mu, \nu) = PH_g(\mu, \nu)^{\text{cut}} +  PH_g(\mu, \nu)^{\text{cj}} + PH_g(\mu, \nu)^{\text{dj}}$.

    Now, we can recall \cref{pruned weights}, where we took the weighted sum over pruned monodromy graphs $\Gamma$ of type $(g, \mu, \nu)\p$ in order to calculate $PH\tr_g(\mu, \nu)$ for $2g-2+\ell(\mu)+\ell(\nu) = k+1$. We may split the set of monodromy graphs $\Gamma \in \mathcal{TB}_g(\mu, \nu) $ into three distinct sets of graphs based on the type of vertex constructed over the final vertex. That is,
    \begin{itemize}
        \item We denote by $\Gamma^\text{cut}$ those pruned monodromy graphs where one regular strand is cut into two over the final vertex, possibly with coloured edges joining this vertex.

        \item Similarly, we denote by $\Gamma^{\text{cj}}$ the subset of these pruned monodromy graphs such that there is a connected join over the last vertex. That is, two regular strands from the same connected component join into one strand at the final vertex, possibly with coloured edges joining from the left.

        \item Finally, we denote by $\Gamma^{\text{dj}}$ those pruned monodromy graphs that have a disconnected join over the final vertex, possibly with coloured edges joining from the left.

    \end{itemize}
    In particular, we have that $\{\  \Gamma \ | \ \Gamma \text{ pruned monodromy graphs of type } (g, \mu, \nu)\p \} = \Gamma^\text{cut} + \Gamma^\text{cj} + \Gamma^\text{dj}$.

    Then, we may take $PH\tr_g(\mu, \nu)^i$ to be the weighted sum over pruned monodromy graphs $\Gamma^i$, where $i \in \{ \text{cut, cj, dj} \}$. Summing these three quantities gives us our total pruned count $PH\tr_g(\mu, \nu)$.
    
    To show that $PH\tr_g(\mu, \nu)^i = PH_g(\mu, \nu)^i$ for $i\in \{\textrm{cut, cj, dj}\}$ we consider the following analysis.
    \begin{itemize}
        \item $PH\tr_g(\mu, \nu)^\text{cut}$ satisfies \cref{pruned c recur}. Indeed, each pruned monodromy graph $\Gamma^\text{cut}$ is obtained from a pruned monodromy graph of type $(g, \mu_I, (\nu \backslash \{\nu_i, \nu_j\}, \alpha))$ by cutting a strand of weight $\alpha$ over the final vertex and joining any coloured edges if $\mu_I \neq \mu$. By the induction step $PH\tr_g(\mu_I, (\nu \backslash \{\nu_i, \nu_j\}, \alpha)) = PH_g(\mu_I, (\nu \backslash \{\nu_i, \nu_j\}, \alpha))$.
        However, $PH\tr_g(\mu, \nu)^\text{cut}$ is calculated by multiplying new edge weights, the new vertex multiplicity, by any new automorphisms, and by the weight of the old graph up to that point. 
        By construction, our vertex multiplicity is exactly the combinatorial factor that features in the recursion. The tropical edge weights correspond to the weights $\alpha$ and $ s$, for $s\in \mu_{I^c}$ in the recursion. Moreover, the automorphism factors of the graphs negate any overcounting of quasi-isomorphic covers, whereas the automorphism factors of the partitions give the multiplicity needed for labelled pruned double Hurwitz numbers. 
        Thus, $PH\tr_g(\mu, \nu)^\text{cut} = PH_g(\mu, \nu)^\text{cut}$.

        \item Similarly, each pruned monodromy graph $\Gamma^\text{cj}$ arises by taking a pruned monodromy graph of type $(g-1, \mu_I, (\nu\backslash \{\nu_i\}, \alpha, \beta))$ and joining regular strands of weight $\alpha, \beta $ over the final vertex, where $\alpha, \beta$ belong to the same connected component, and join any coloured strands of weight $s$, $s\in \mu_{I^c}$ if $I^c \neq \emptyset$.
        Analogously, $PH\tr_g(\mu, \nu)^\text{dj}$ satisfies \cref{pruned cj recur}. Again, this follows by construction, with the combinatorial factor in the recursion appearing as the vertex multiplicity in our graphs, and with $PH_{g-1}( \mu_I, (\nu\backslash \{\nu_i\}, \alpha, \beta)) = PH\tr_{g-1}(\mu_I, (\nu\backslash \{\nu_i\}, \alpha, \beta))$ by the induction hypothesis. Furthermore, the automorphism factors count as described above by cancelling any overcounting of quasi-isomorphic graphs and ensuring the right count for labelled covers.
        Thus, $ PH\tr_g(\mu, \nu)^\text{cj} = PH_g(\mu, \nu)^\text{cj}$.

        \item Finally, pruned monodromy graphs $\Gamma^\text{dj}$ are formed by taking pruned monodromy graphs of type $(g_1, \mu_{I_1}, (\nu_{J_1}, \alpha)), (g_2, \mu_{I_2}, (\nu_{J_2}, \beta))$ and joining regular strands of weight $\alpha, \beta $ over the final vertex, along with any coloured edges. $PH\tr_g(\mu, \nu)^\text{dj}$ satisfies \cref{pruned dj recur}. Indeed, the weights and the vertex multiplicity are defined this way by construction, where we do not take into consideration the ordering of the vertices in the different connected components, as we take quasi-isomorphic graphs. We find that $PH\tr_g(\mu, \nu)^\text{dj} = PH_g(\mu, \nu)^\text{dj}$.
    \end{itemize}

    Putting these three parts together, we find our desired equivalence
    $$PH\tr_g(\mu, \nu) = PH_g(\mu, \nu).$$

    \end{proof}

    To end this subsection, we give an example computing a pruned double Hurwitz number, using this correspondence theorem.

\begin{figure}[ht]
    \centering
    \includegraphics[width=\linewidth]{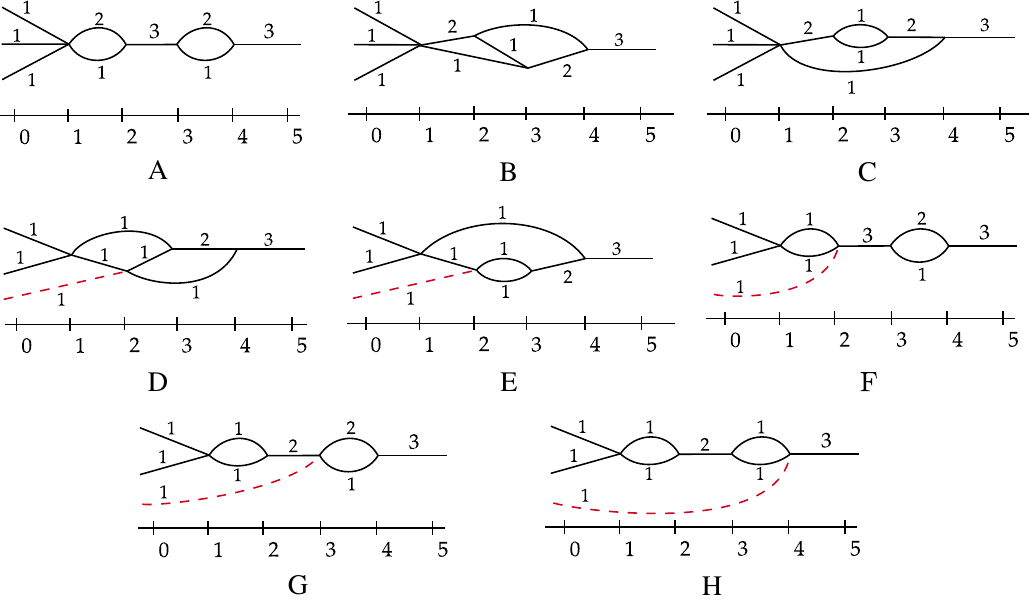}
    \caption{All pruned monodromy graphs of type $(2, (1, 1, 1), (3))\p$.}
    \label{fig:bigprunes}
\end{figure}

\begin{example}
    Let us consider $PH\tr_2((1, 1, 1), (3))$. 
    We can see all monodromy graphs of type $(2, (1, 1, 1), (3))\p$ are as shown in \cref{fig:bigprunes}.
%
    The initial vertex multiplicities are the local pruned double Hurwitz numbers
    $$ PH_0 ((1, 1, 1), (2, 1)) = 6, \qquad PH_0 ((1, 1), (1, 1)) = 2.$$
    Furthermore, the vertex multiplicity of any secondary vertex $v$, to which no coloured edges join is $m(v)= 1$. 
    Then, we find the weights of our graphs in \cref{TablePrunes}.

    \begin{table}[ht]
    \begin{tabular}{|c| c c c c | c |} 
    \hline 
    Graph $\Gamma$ & $\frac{|\Aut(\mu)||\Aut(\nu)|}{|\Aut (\Gamma )|}$ & $ \prod_{e \text{ inner edge}} w(e)$  & $ \prod_{\hat{e} \text{ col. edge}} w(\hat{e})$ & $\prod_{v \text{ vertex}} m(v)$ & Total \\ [0.5ex] 
    \hline\hline
    \textrm{A} & $6 \cdot (1/6)$ & $2\cdot 3\cdot 2$ & $1$ & $6 $ & $72$  \\ 
 
    \hline
    \textrm{B} & $6 \cdot (1/6)$ & $2\cdot 2$ & $1$ & $6$ & $24$ \\ 
 
    \hline
    \textrm{C} & $6 \cdot (1/6)\cdot (1/2)$ & $2\cdot 2$ & $1$ & $6$ & $12$ \\ 
 
    \hline
    \textrm{D} & $6 \cdot (1/2)$ & $2$ & $1$ & $2\cdot (3!/2!) \cdot 2!$ & $72$ \\ 
 
    \hline
    \textrm{E} & $6 \cdot (1/2)\cdot(1/2)$ & $2$ & $1$ & $2\cdot (3!/2!) \cdot 2!$ & $36$ \\ 

    \hline
    \textrm{F} & $6 \cdot (1/2)\cdot (1/2)$ & $3\cdot 2$ & $1$ & $2\cdot (3!/2!)\cdot 2$ & $108$ \\ 

    \hline
    \textrm{G} & $6 \cdot (1/2)\cdot (1/2)$ & $2\cdot 2$ & $1$ & $2\cdot (4!/3!) \cdot 2!$ & $96$ \\ 
 
    \hline
    \textrm{H} & $6 \cdot (1/2) \cdot (1/2)\cdot (1/2)$ & $2$ & $1$ & $2\cdot (5!/4!) \cdot 2$  & $30$ \\ [1ex] 
    \hline
    
    \end{tabular} 
    \caption{\label{TablePrunes} The weight for each pruned monodromy graph of type $(2, (1, 1, 1), (3))\p$.}
    \end{table}

    Summing the total column gives us $$ PH\tr_2((1, 1, 1), (3)) = 450  = PH_2((1, 1, 1), (3))$$.
\end{example}

\subsection{Polynomiality in genus $0$}
\label{sec:polytrop}
In this subsection, we study the polynomiality of pruned double Hurwitz numbers in genus $0$ from a tropical perspective. Based on this we take first steps towards wall--crossing formulae for pruned double Hurwitz numbers in genus $0$.

We begin with the notion of the combinatorial type of a pruned monodromy graph.

\begin{definition}
    Let $\pi\colon\Gamma\to[0,\tilde{b}+1]$ a pruned monodromy graph as in \cref{def:prunedmonod}. The map $\pi$ induces a partial ordering $\mathcal{O}$ on the edge set of $\Gamma$. We call the tuple $(\Gamma,\mathcal{O})$, where we forgot the weights of $\pi$ but retain the information in \cref{def:prunedmonod}, the combinatorial type of $\pi$.
\end{definition}

When there is no potential for confusion, we will denote the combinatorial type of $\pi$ by $\Gamma$.\vspace{\baselineskip}

Recall the set $\mathfrak{TP}_g(\mu,\nu)$ of pruned monodromy graphs of type $(g,\mu,\nu)\p$. Let $m$ and $n$ be positive integers, then we denote by $\mathfrak{TP}_g(m,n)$ the set of all tuples $(\Gamma,\mathcal{O})$, such that there exist partitions $\mu,\nu$ of length $m$ and $n$, so that $(\Gamma,\mathcal{O})$ is the combinatorial type of pruned monodromy graph in $\mathfrak{TP}_g(\mu,\nu)$. Since by definition the number of vertices and edges of a pruned monodromy graph is bounded in terms of $g$, $m$ and $n$, the set $\mathfrak{TP}_g(m,n)$ is finite.\\
We now fix $g=0$ for the rest of our discussion. Since in this case, $\Gamma$ is a tree, fixing the weights $\mu_1,\dots,\mu_m$ and $\nu_1,\dots,\nu_n$ of the strands over $0$ and $\infty$ also determines the weights of all inner edges by the balancing condition \cref{def:prunedmonod} (5). By the arguments as in \cite[Lemma 6.4]{cavalieri2010tropical}, the weight of an edge is a linear polynomial in the $\mu_i$ and $\nu_j$. More precisely, we have for an edge $e$ of $\Gamma$ that
\begin{equation}
    \omega(e)=\sum_{i\in I}\mu_i-\sum_{j\in J}\nu_j
\end{equation}
where $I\subset[m]$ and $J\subset[n]$ are the in- and out-weights of the component of $\Gamma\backslash \{e\}$ from which $e$ points away.\\
Now, we observe that the vertex multiplicities of secondary vertices in \cref{pruned weights} only depends on the combinatorial type of a pruned monodromy graph. Moreover, the multiplicity of the primary vertices are exactly the pruned double Hurwitz numbers considered as the third case of \cref{base case pruned ex}, which were observed to be piecewise polynomial with respect to the resonance arrangement.\\
We note, however, that here we need to introduce a refinement of the resonance arrangement in order to deduce piecewise polynomiality from the tropical perspective.

Recall the hyperplane
\begin{equation}
    \mathcal{H}_{m,n}=\left\{(\mu,\nu)\in\mathbb{N}^m\times\mathbb{N}^n\mid\sum\mu_i=\sum\nu_j\right\}.
\end{equation}
and the \textit{resonance arrangement} in $\mathcal{H}_{m,n}$ given by
\begin{equation}
    \mathcal{R}_{m,n}=\left\{\sum_{i\in I}\mu_i-\sum_{j\in J}\nu_j=0\mid I\subset[m],J\subset[n]\right\}.
\end{equation}

\begin{definition}
    Let $m$, $n$ positive integers, then we define the refined resonance arrangement in $\mathcal{H}_{m,n}$ given by
\begin{equation}
    \tilde{\mathcal{R}}_{m,n}=\left\{\sum_{i\in I}\alpha(i)\mu_i-\sum_{j\in J}\beta(j)\nu_j=0\mid I\subset[m],J\subset[n],\alpha(i),\beta(i)\in\{-1,1\}\right\}.
\end{equation}
\end{definition}

Note, that the vertex multiplicity of any primary vertex of a tropical monodromy graph in $\mathfrak{TP}_g(\mu,\nu)$ is given by
\begin{equation}
    PH_0\left((\mu_{i_1},\dots,\mu_{i_s}),\left(\sum_{I_1}\mu_i-\sum_{J_1}\nu_j,\sum_{I_2}\mu_i-\sum_{J_2}\nu_j\right)\right)
\end{equation}
for some subsets $I_i\subset[m],J_j\subset[n]$. Substituting the partition $(\sum_{I_1}\mu_i-\sum_{J_1}\nu_j,\sum_{I_2}\mu_i-\sum_{J_2}\nu_j)$ into the equations of the resonance arrangement, we obtain equations of the form that define $\tilde{\mathcal{R}}_{m,n}$. Thus, we have proved the following.

\begin{theorem}
    Let $(\Gamma,\mathcal{O})\in \mathfrak{TP}_0(m,n)$ and denote by $m(\Gamma,\mathcal{O})(\mu,\nu)$ the contribution of pruned monodromy graphs in $\mathfrak{TP}_0(\mu,\nu)$ to $PH_0(\mu,\nu)$. Then, the map
    \begin{align}
        m(\Gamma,\mathcal{O})\colon\mathcal{H}_{m,n}&\to\mathbb{Q}\\
        (\mu,\nu)&\mapsto m(\Gamma,\mathcal{O})(\mu,\nu) 
    \end{align}
    is piecewise polynomial with respect to the refined resonance arrangement $\tilde{\mathcal{R}}_{m,n}$.
\end{theorem}

We remark that in all examples we computed, the map $m(\Gamma,\mathcal{O})$ was actually piecewise polynomial with respect to $\mathcal{R}_{m,n}$. We illustrate this phenomenon in the following example.

\begin{example}
    Let $g=0$, $\mu=(\mu_1,\mu_2)$ and $\nu=(\nu_1,\nu_2,\nu_3)$. We consider the combinatorial type of a pruned monodromy graph in \cref{fig:prungraphpoly}. We consider the chamber of the resonance arrangement given by $\mu_1,\mu_2>\nu_1,\nu_2,\nu_3$. Note that these inequalities automatically imply $\nu_i+\nu_j>\mu_1,\mu_2$ for all $i\neq j$. Its multiplicity is
    \begin{equation}
        2\mathrm{min}(\mu_1,\mu_2,\mu_1+\mu_2-\nu_1,\nu_1)(\mu_1+\mu_2-\nu_1).
    \end{equation}
    Since $\mu_1,\mu_2>\nu_1$, we need to compute $\mathrm{min}(\mu_1+\mu_2-\nu_1,\nu_1)$. We observe that $\mu_1+\mu_2-\nu_1>\nu_1$ is equivalent to $\mu_1+\mu_2>2\nu_1$ which is true in the chosen chamber. Therefore, the contribution of the graph in \cref{fig:prungraphpoly} to $PH_0((\mu_1,\mu_2),(\nu_1,\nu_2,\nu_3))$ is
    \begin{equation}
        2\nu_1(\mu_1+\mu_2-\nu_1).
    \end{equation}
    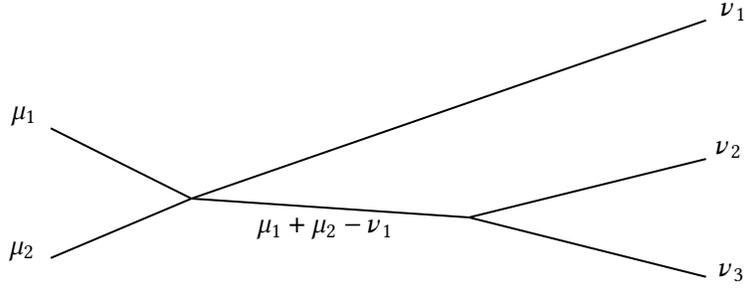
\begin{figure}
        \centering
\tikzset{every picture/.style={line width=0.75pt}} 

\begin{tikzpicture}[x=0.75pt,y=0.75pt,yscale=-1,xscale=1]

\draw    (100,105) -- (171,140.5) ;
\draw    (171,140.5) -- (100,170.5) ;
\draw    (171,140.5) -- (430.5,50.5) ;
\draw    (171,140.5) -- (311,150) ;
\draw    (311,150) -- (430.5,120.5) ;
\draw    (311,150) -- (430.5,180) ;

\draw (77.5,160.5) node [anchor=north west][inner sep=0.75pt]   [align=left] {$\displaystyle \mu _{2}$};
\draw (78.5,92.5) node [anchor=north west][inner sep=0.75pt]   [align=left] {$\displaystyle \mu _{1}$};
\draw (436,172) node [anchor=north west][inner sep=0.75pt]   [align=left] {$\displaystyle \nu _{3}$};
\draw (434.5,110) node [anchor=north west][inner sep=0.75pt]   [align=left] {$\displaystyle \nu _{2}$};
\draw (437,40.5) node [anchor=north west][inner sep=0.75pt]   [align=left] {$\displaystyle \nu _{1}$};
\draw (202.5,149) node [anchor=north west][inner sep=0.75pt]   [align=left] {$\displaystyle \mu _{1} +\mu _{2} -\nu _{1}$};

\end{tikzpicture}
        \caption{A pruned monodromy graph of type $(0,(\mu_1,\mu_2),(\nu_1,\nu_2,\nu_3))$.}
        \label{fig:prungraphpoly}
    \end{figure}
\end{example}

\printbibliography

@article {cavalieri2010tropical,
	AUTHOR = {Cavalieri, Renzo and Johnson, Paul and Markwig, Hannah},
	TITLE = {Tropical {H}urwitz numbers},
	JOURNAL = {J. Algebraic Combin.},
	FJOURNAL = {Journal of Algebraic Combinatorics. An International Journal},
	VOLUME = {32},
	YEAR = {2010},
	NUMBER = {2},
	PAGES = {241--265},
	MRCLASS = {14N10 (14T05)},
	MRNUMBER = {2661417},
	MRREVIEWER = {Ethan G. Cotterill},
}

@article{irving2009minimal,
  title={Minimal factorizations of permutations into star transpositions},
  author={Irving, John and Rattan, Amarpreet},
  journal={Discrete Mathematics},
  volume={309},
  number={6},
  pages={1435--1442},
  year={2009},
  publisher={Elsevier}
}

@article{zvonkine2004algebra,
  title={An algebra of power series arising in the intersection theory of moduli spaces of curves and in the enumeration of ramified coverings of the sphere},
  author={Zvonkine, Dimitri},
  journal={arXiv preprint math/0403092},
  year={2004}
}

@article{cavalieri2016tropicalizing,
  title={Tropicalizing the space of admissible covers},
  author={Cavalieri, Renzo and Markwig, Hannah and Ranganathan, Dhruv},
  journal={Mathematische Annalen},
  volume={364},
  pages={1275--1313},
  year={2016},
  publisher={Springer}
}

@inproceedings{cavalieri2018graphical,
  title={A graphical interface for the Gromov-Witten theory of curves},
  author={Cavalieri, Renzo and Johnson, Paul and Markwig, Hannah and Ranganathan, Dhruv},
  booktitle={Algebraic geometry: Salt Lake City 2015},
  pages={139--167},
  year={2018}
}

@article{hahn2020bi,
  title={Bi-pruned Hurwitz numbers},
  author={Hahn, Marvin Anas},
  journal={Journal of Combinatorial Theory, Series A},
  volume={174},
  pages={105240},
  year={2020},
  publisher={Elsevier}
}

@Article{zbMATH06791415,
 Author = {Hahn, Marvin Anas},
 Title = {Pruned double {Hurwitz} numbers},
 FJournal = {The Electronic Journal of Combinatorics},
 Journal = {Electron. J. Comb.},
 Volume = {24},
 Number = {3},
 Pages = {research paper p3.66, 32},
 Year = {2017},
 Language = {English},
 zbMATH = {6791415},
 Zbl = {1378.14051}
}

@article{do2018pruned,
  title={Pruned Hurwitz numbers},
  author={Do, Norman and Norbury, Paul},
  journal={Transactions of the American Mathematical Society},
  volume={370},
  number={5},
  pages={3053--3084},
  year={2018}
}

@article{johnson2015double,
  title={Double Hurwitz numbers via the infinite wedge},
  author={Johnson, Paul},
  journal={Transactions of the American Mathematical Society},
  volume={367},
  number={9},
  pages={6415--6440},
  year={2015}
}

@article{cavalieri2011wall,
  title={Wall crossings for double Hurwitz numbers},
  author={Cavalieri, Renzo and Johnson, Paul and Markwig, Hannah},
  journal={Advances in Mathematics},
  volume={228},
  number={4},
  pages={1894--1937},
  year={2011},
  publisher={Elsevier}
}

@article{shadrin2008chamber,
  title={Chamber behavior of double Hurwitz numbers in genus 0},
  author={Shadrin, Sergei and Shapiro, Michael and Vainshtein, Alek},
  journal={Advances in Mathematics},
  volume={217},
  number={1},
  pages={79--96},
  year={2008},
  publisher={Elsevier}
}

@Article{zbMATH01539310,
 Author = {Okounkov, Andrei},
 Title = {Toda equations for {Hurwitz} numbers},
 FJournal = {Mathematical Research Letters},
 Journal = {Math. Res. Lett.},
 Volume = {7},
 Number = {4},
 Pages = {447--453},
 Year = {2000},
 Language = {English},
 zbMATH = {1539310},
 Zbl = {0969.37033}
}

@article{ekedahl2000hurwitz,
  title={Hurwitz numbers and intersections on moduli spaces of curves},
  author={Ekedahl, Torsten and Lando, Sergei and Shapiro, Michael and Vainshtein, Alek},
  journal={arXiv preprint math/0004096},
  year={2000}
}

@article{ekedahl1999hurwitz,
  title={On Hurwitz numbers and Hodge integrals},
  author={Ekedahl, Torsten and Lando, Sergei and Shapiro, Michael and Vainshtein, Alek},
  journal={Comptes Rendus de l'Acad{\'e}mie des Sciences-Series I-Mathematics},
  volume={328},
  number={12},
  pages={1175--1180},
  year={1999},
  publisher={Elsevier}
}

@article {G,
	AUTHOR = {Gunningham, Sam},
	TITLE = {Spin {H}urwitz numbers and topological quantum field theory},
	JOURNAL = {Geom. Topol.},
	FJOURNAL = {Geometry \& Topology},
	VOLUME = {20},
	YEAR = {2016},
	NUMBER = {4},
	PAGES = {1859--1907},
	MRCLASS = {81T45},
	MRNUMBER = {3548460},

}

@article{duchi2014bijections,
  title={Bijections for simple and double Hurwitz numbers},
  author={Duchi, Enrica and Poulalhon, Dominique and Schaeffer, Gilles},
  journal={arXiv preprint arXiv:1410.6521},
  year={2014}
}

@article{borot2011matrix,
  title={A matrix model for simple Hurwitz numbers, and topological recursion},
  author={Borot, Ga{\"e}tan and Eynard, Bertrand and Mulase, Motohico and Safnuk, Brad},
  journal={Journal of Geometry and Physics},
  volume={61},
  number={2},
  pages={522--540},
  year={2011},
  publisher={Elsevier}
}

@article{eynard2011laplace,
  title={The Laplace transform of the cut-and-join equation and the Bouchard--Marino conjecture on Hurwitz numbers},
  author={Eynard, Bertrand and Mulase, Motohico and Safnuk, Bradley},
  journal={Publications of the Research Institute for Mathematical Sciences},
  volume={47},
  number={2},
  pages={629--670},
  year={2011}
}

@InCollection{zbMATH05380331,
 Author = {Bouchard, Vincent and Mari{\~n}o, Marcos},
 Title = {Hurwitz numbers, matrix models and enumerative geometry},
 BookTitle = {From Hodge theory to integrability and TQFT tt*-geometry. International workshop From TQFT to tt* and integrability, Augsburg, Germany, May 25--29, 2007},
 ISBN = {978-0-8218-4430-4},
 Pages = {263--283},
 Year = {2008},
 Publisher = {Providence, RI: American Mathematical Society (AMS)},
 Language = {English},
 zbMATH = {5380331},
 Zbl = {1151.14335}
}

@article{hurwitz1891riemann,
  title={{\"U}ber Riemann'sche Fl{\"a}chen mit gegebenen Verzweigungspunkten},
  author={Hurwitz, Adolf},
  journal={Mathematische Annalen},
  volume={39},
  pages={1--60},
  year={1891},
  publisher={Springer}
}

@article{goulden2005towards,
  title={Towards the geometry of double Hurwitz numbers},
  author={Goulden, Ian P and Jackson, David M and Vakil, Ravi},
  journal={Advances in Mathematics},
  volume={198},
  number={1},
  pages={43--92},
  year={2005},
  publisher={Elsevier}
}

@article{dyckmodk,
  
	year = 2023,
	month = {feb},
  
	publisher = {The Electronic Journal of Combinatorics},
  
	volume = {30},
  
	number = {1},
  
	author = {Clemens Heuberger and Sarah J. Selkirk and Stephan Wagner},
  
	title = {Enumeration of Generalized Dyck Paths Based on the Height of Down-Steps Modulo $k$},
  
	journal = {The Electronic Journal of Combinatorics}
}

@article{dycknondecreasing,
title = {Nondecreasing Dyck paths and q-Fibonacci numbers},
journal = {Discrete Mathematics},
volume = {170},
number = {1},
pages = {211-217},
year = {1997},
author = {E. Barcucci and A. {Del Lungo} and S. Fezzi and R. Pinzani}}

@article{dycknondecreasingenum,
author = {Czabarka, Eva and Flórez, Rigoberto and Junes, Leandro},
year = {2015},
month = {01},
pages = {},
title = {Some Enumerations on Non-Decreasing Dyck Paths},
volume = {22},
journal = {The electronic journal of combinatorics},
}

@article{dyckcoloured,
title = {Dyck paths with coloured ascents},
journal = {European Journal of Combinatorics},
volume = {29},
number = {5},
pages = {1262-1279},
year = {2008},
author = {Andrei Asinowski and Toufik Mansour}}
\end{document}